\documentclass[11pt]{article}

\RequirePackage[OT1]{fontenc}
\RequirePackage{amsthm,amsmath,natbib}
\RequirePackage[colorlinks,citecolor=blue,urlcolor=blue]{hyperref}

\usepackage{geometry}
\usepackage{graphicx}
\usepackage[caption = false]{subfig}
\usepackage{amssymb, amsfonts}  
\usepackage{enumerate, color, framed, multirow}
\usepackage{comment, longtable, appendix}
\usepackage{placeins}

\binoppenalty=\maxdimen
\relpenalty=\maxdimen

\allowdisplaybreaks

\newcommand{\ds}{\displaystyle}
\newcommand{\E}{\text{E}}

\newcommand{\X}{\mathsf{X}}
\newcommand{\B}{\mathcal{B}}
\newcommand{\real}{{\mathbb R}}
\newcommand{\N}{{\mathbb N}}

\newcommand\numberthis{\addtocounter{equation}{1}\tag{\theequation}}

\numberwithin{equation}{section}
\theoremstyle{plain}
\newtheorem{thm}{Theorem}

\newtheorem{lemma}{Lemma}
\newtheorem{corollary}{Corollary}

 \theoremstyle{remark}
\newtheorem{rem}{Remark}
\newtheorem{cond}{Condition}

\author{ Dootika Vats\\ School of Statistics\\ University of
  Minnesota\\ \texttt{vatsx007@umn.edu} \and James Flegal
  \thanks{Research supported by the National Science
    Foundation.}\\ Department of Statistics\\ University of
  California, Riverside\\ \texttt{jflegal@ucr.edu} \and Galin Jones
  \thanks{Research supported by the National Institutes of Health and
    the National Science Foundation.}\\ School of
  Statistics\\ University of Minnesota\\ \texttt{galin@umn.edu} }
\title{Strong Consistency of Multivariate Spectral Variance
  Estimators} \date{\today}

\begin{document}

\maketitle

\begin{abstract}
  Markov chain Monte Carlo (MCMC) algorithms are used to estimate
  features of interest of a distribution. The Monte Carlo error in
  estimation has an asymptotic normal distribution whose multivariate
  nature has so far been ignored in the MCMC community. We present a
  class of multivariate spectral variance estimators for the
  asymptotic covariance matrix in the Markov chain central limit
  theorem and provide conditions for strong consistency. We examine
  the finite sample properties of the multivariate spectral variance
  estimators and its eigenvalues in the context of a vector
  autoregressive process of order 1.
\end{abstract}

\section{Introduction} 
\label{sec:introduction}

Markov chain Monte Carlo (MCMC) methods are often required for
parameter estimation in the statistical models encountered in modern
applications. The typical MCMC experiment consists of simulating a
Markov chain in order to estimate a vector of quantities, such as
moments or quantiles, associated with the target
distribution. However, the multivariate nature of the estimation has
only rarely been acknowledged in the MCMC literature. We consider the
situation where estimation of a vector of means is of interest.  Given
a multivariate Markov chain central limit theorem (CLT) for the sample
mean vector, we show that a class of multivariate spectral variance
estimators (MSVEs) are strongly consistent estimators of the
covariance matrix in the asymptotic normal distribution. We also
establish strong consistency of the eigenvalues of any strongly
consistent estimator of the asymptotic covariance matrix.

We know of no other comparable work in the context of MCMC.
\citet{koso:2000} did propose estimators of the asymptotic covariance
matrix which generalized work in the univariate case by
\citet{geye:1992}. However, these estimators are asymptotically
conservative and are based on the properties of reversible Markov
chains, an assumption we do not make.  There has been a substantial
amount of work in the univariate setting.  In particular,
\citet{atch:2011} and \citet{fleg:jone:2010} established strong
consistency of certain univariate spectral variance estimators, but
the multivariate problem is more complicated and requires much new
work.  Moreover, our work represents a substantial generalization of
the univariate results and requires much weaker conditions on the
Markov chain.  Thus we also improve the current results in the
univariate setting.

We will give a more formal description of the problem studied here.
Let $F$ be a probability distribution with support $\X$, equipped with
a countably generated $\sigma$-field $\B(\X)$ and let $g: \X \to
\mathbb{R}^p$ be an $F$-integrable function such that
\[
\theta := \E_{F}g = \int_{\X} g(x) \, dF
\]
is the $p$-dimensional vector of interest. Note that $\X$ and $\theta$
often have different dimensions.  It is common to resort to MCMC
methods to estimate $\theta$ when it is difficult to obtain $\theta$
analytically or to produce independent samples from $F$.
MCMC is popular because it is straightforward
to simulate a Harris ergodic (i.e., aperiodic, $F$-irreducible, and
positive Harris recurrent) Markov chain having invariant distribution
$F$ \citep{geye:2011, liu:2008, robe:case:2013}.  Letting $X = \{X_1,
X_2, X_3, \dots \}$ denote such a Markov chain, estimation is easy
since, for any initial distribution, with probability 1,
 \begin{equation}
 \label{eq:strong.lln}
\theta_n := \dfrac{1}{n} \ds \sum_{t=1}^{n}g(X_t) \to \theta \quad \text{ as } \quad n \to \infty \; .
\end{equation}
Of course, for any $n$ there will be an unknown \textit{Monte Carlo
  error} in estimation, $\theta_n - \theta$, and assessment of this
Monte Carlo error is critical to the reliability of the simulation
results \citep{fleg:hara:jone:2008, fleg:jone:2011, geye:1992,
  jone:hobe:2001}.  However, the multivariate nature of the Monte
Carlo error has been largely ignored in the MCMC literature \citep[but
  see][]{gong:fleg:2015}.

  Instead, the primary focus has been on assessing the univariate
  Monte Carlo error.  Let $g^{(i)}$, $\theta_n^{(i)}$, and
  $\theta^{(i)}$, denote the $i$th components of $g$, $\theta_{n}$,
  and $\theta$, respectively. Then $\theta_n^{(i)} - \theta^{(i)}$ is
  the unknown Monte Carlo error of the $i$th component.  The
  approximate sampling distribution of this error is available via a Markov
  chain CLT if there exists $0<\sigma^2_{i} < \infty$ such that, as $n
  \to \infty$,
\begin{equation}
\label{eq:uni_clt}
	\sqrt{n}(\theta_n^{(i)} - \theta^{(i)}) \stackrel{d}{\to}\text{N}(0, \sigma^2_{i}) \; .
\end{equation}
(See \citet{jone:2004} and \citet{robe:rose:2004} for a discussion of
the conditions for \eqref{eq:uni_clt}.)  Due to serial
correlation in $X$, $\text{Var}_{F} g^{(i)} \neq
\sigma^2_{i}$, except in trivial cases.  Nevertheless, consistent
estimation of $\sigma^2_{i}$ is key to constructing asymptotically
valid confidence intervals for $\theta^{(i)}$ and hence in assessing
the reliability of the simulation results \citep{fleg:gong:2015,
  fleg:hara:jone:2008, glyn:whit:1992, jone:hara:caff:neat:2006,
  jone:hobe:2001}.  Thus consistent estimation of $\sigma^2_{i}$ has
received significant attention; \citet{atch:2011}, \citet{dame:1991},
and \citet{fleg:jone:2010} studied spectral variance estimators,
\citet{hobe:jone:pres:rose:2002} and \citet{mykl:tier:yu:1995}
investigated estimators based on regenerative simulation, and
\citet{jone:hara:caff:neat:2006} studied nonoverlapping batch means.
\citet{geye:1992} introduced asymptotically conservative estimators
based on the spectral properties of reversible Markov chains.
\citet{doss:fleg:jone:neat:2014} considered univariate estimators in
the context of estimating quantiles.

In the multivariate setting, the approximate sampling distribution of
the Monte Carlo error is available via a Markov chain CLT if there
exists a positive definite $p \times p$ matrix $\Sigma$ such that
\begin{equation}
  \label{eq:multi_clt}
	\sqrt{n}(\theta_n - \theta) \overset{d}{\to} \text{N}_p (0, \Sigma) \quad \text{ as } \quad n \to \infty \; .
\end{equation}
We consider a class of MSVEs of $\Sigma$ and provide conditions for
strong consistency.  Our main assumption on the process is the
existence of a multivariate \textit{strong invariance principle}
(SIP); that is, we assume that the centered and appropriately scaled
partial sum process is similar to a Brownian motion. Specifically, an SIP
holds for $\{g(X_{t})\}_{t\geq1}$ if there exists a $p \times p$ lower
triangular matrix $L$ and an increasing function $\psi$ on the
integers such that, with probability 1,
\[ 
 n(\theta_{n} - \theta) = LB(n) + O(\psi(n)) ~~~\text{ as } ~~ n \to \infty~,
\]
where $B(n)$ denotes a $p$-dimensional standard Brownian motion and
$LL^T = \Sigma$. If $\psi$ is such that $\psi(n)/\sqrt{n} \to 0$ as $n
\to \infty$, the SIP implies a strong law, a CLT, and a functional CLT
for $\theta_{n}$. Under moment conditions on $g$, an SIP with $\psi(n)
= n^{1/2 - \lambda}$ for some $ \lambda > 0 $ holds for polynomially
ergodic Markov chains.

There has been a substantial amount of work in the context of MCMC on
establishing that Markov chains are at least polynomially ergodic.  An
incomplete list is given by \citet{acos:hube:jone:2015},
\citet{doss:hobe:2010}, \citet{fort:moul:2003},
\citet{hobe:geye:1998}, \citet{jarn:hans:2000},
\citet{jarn:robe:2002}, \citet{jarn:robe:2007},
\citet{john:geye:2012}, \citet{john:jone:2015},
\citet{jone:robe:rose:2014}, \citet{marc:hobe:2004},
\citet{petr:robe:rose:1999}, \citet{robe:rose:1999a},
\citet{robe:twee:1996}, \citet{rose:1996}, \citet{roy:hobe:2007},
\citet{tan:hobe:2012}, \citet{tan:jone:hobe:2013}, and
\citet{tier:1994}.  While establishing that a Markov chain is at least
polynomially ergodic can be challenging, it is not the obstacle that
it once was.

\subsection{Motivating Example} 
\label{sub:motivating_example}

As motivation for the use of multivariate methods, we present a simple Bayesian logistic regression model. For $i = 1, \dots, K$, let $Y_i$ be a binary response variable. For the $i$th observation let $X_i = (x_{i1}, x_{i2}, \dots, x_{i5})$ be the observed vector of predictors, then
\begin{equation}\label{eq:logistic model}
Y_i | X_i, \beta  \overset{ind}{\sim} \text{Bernoulli} \left( \dfrac{1}{1 + e^{-X_i \beta}}  \right)\, ,~~~\text{ and } ~~~\beta  \sim N_5(0, I_5)\; .
\end{equation}
The resulting posterior $F$ is intractable and hence MCMC is used to obtain estimates of the regression coefficient, $\beta$. We use the \texttt{logit} dataset in the \texttt{mcmc} R package which contains four predictors and 100 observations. The goal is to estimate the posterior mean of $\beta = (\beta_0, \beta_1, \beta_2, \beta_3, \beta_4)^T$. Thus $g$ here is the identity function mapping to $\mathbb{R}^5$.

To sample from the posterior we use the Polya-Gamma Gibbs sampler of \cite{pol:scot:win:2013} (see the \texttt{R} package \texttt{BayesLogit}) which was shown to be uniformly ergodic by \cite{cho:hob:2013}. Although the chain mixes fairly quickly as seen in the  autocorrelation plot for $\beta_0$ in Figure \ref{fig:polya_acf}, the cross-correlation plot between $\beta_0$ and $\beta_2$ indicates correlation across these components that is ignored by univariate methods.  As a result in Figure \ref{fig:polya_conf}, the multivariate confidence ellipse is oriented along non-standard axes (see \cite{vat:fle:jon:2015} for details on how to construct such confidence regions). The ellipse is compared to two univariate confidence boxes; the smaller uncorrected for multiple testing and the larger corrected for two tests using a Bonferroni correction. 

\begin{figure} 
\centering
  \includegraphics[width = 4in]{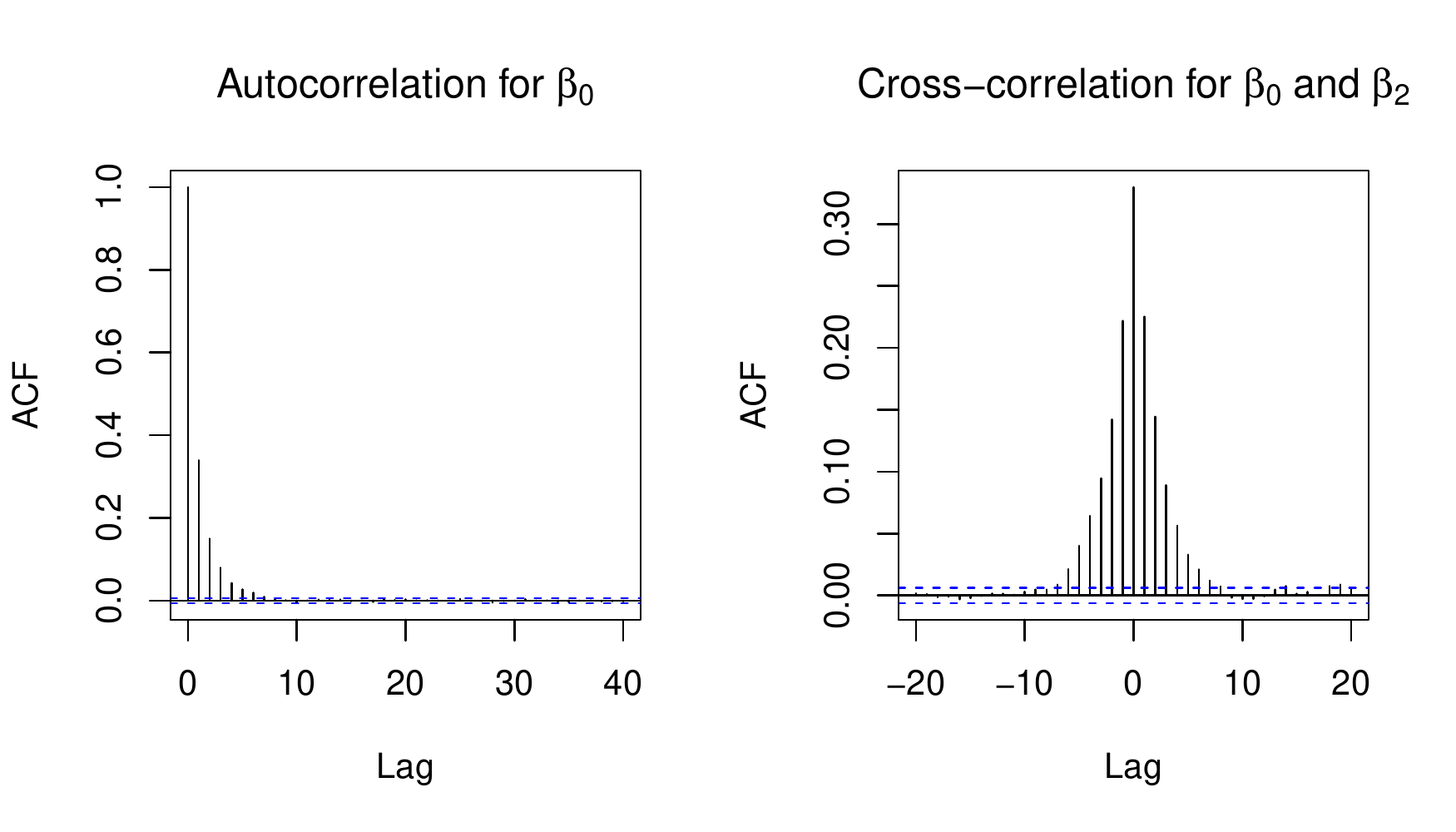}
  \caption[]{Autocorrelation plot for $\beta_0$ and the cross-correlation plot between $\beta_0$ and $\beta_2$ for a Monte Carlo sample size of $n = 10^5$.}
  \label{fig:polya_acf}
\end{figure}

\begin{figure}[tb]
\centering
  \includegraphics[width = 2.5in]{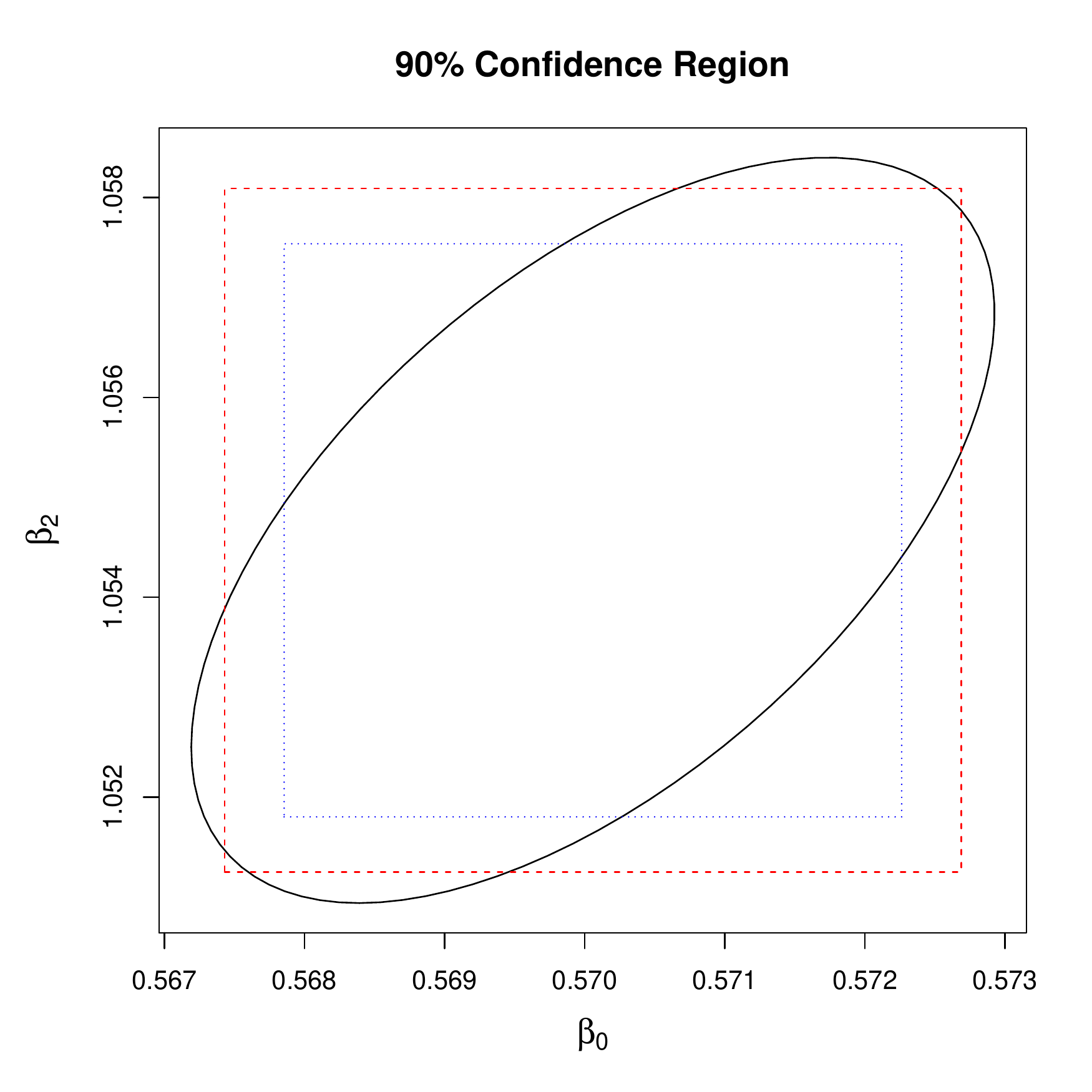}
  \caption{90\% confidence regions constructed using univariate and multivariate methods. The solid ellipse is constructed using an MSVE, the dotted smaller box is constructed using an uncorrected univariate spectral variance estimator and the dashed larger box is constructed using a univariate spectral variance estimator corrected by Bonferroni.}
  \label{fig:polya_conf}
\end{figure}

We assess the performance of these confidence regions by comparing their coverage probabilities and volumes over 1000 independent replications for varying Monte Carlo sample sizes. In particular we look at the volume to the $p$th root ($p$ = 5 in this example). The `true' posterior mean is determined by obtaining a Monte Carlo estimate from a sample of length $10^9$. Results are presented in Table \ref{tab:polya_coverage}. Note that as the Monte Carlo sample size increases, the multivariate methods produce confidence regions with the nominal coverage probability of $90\%$ with significantly lower volume compared to the Bonferroni corrected regions. The uncorrected regions have far from desirable coverage probabilities.

\begin{table}[tb]
  \centering
  \caption{Volume to the $p$th ($p=5$) root and coverage probabilities for 90\% confidence regions constructed using MSVE, uncorrected univariate spectral estimators  and Bonferroni corrected univariate spectral estimators. Replications = 1000 and standard errors are indicated in parenthesis.}
    \label{tab:polya_coverage}
  \begin{tabular}{c|ccc}
  \hline
    $n$ & MSVE & Bonferroni corrected & Uncorrected \\
  \hline

  \multicolumn{4}{c}{Volume to the $5$th root} \\ 
  \hline

  1e3 & 0.0574 \tiny{(4.93e-05)} &  0.0687 \tiny{(7.02e-05)} & 0.0483 \tiny{(4.93e-05)}\\
  1e4 & 0.0189 \tiny{(7.50e-06)} &  0.0226 \tiny{(1.12e-05)} & 0.0160 \tiny{(7.90e-06)}\\
  1e5 & 0.0061 \tiny{(1.10e-06)} &  0.0073 \tiny{(1.50e-06)} & 0.0051 \tiny{(1.10e-06)}\\
  \hline
  \multicolumn{4}{c}{Coverage Probabilities} \\ 
  \hline
  1e3 & 0.853 \tiny{(0.0112)} &  0.871 \tiny{(0.0106)} & 0.549 \tiny{(0.0157)}\\ 
  1e4 & 0.882 \tiny{(0.0102)} &  0.904 \tiny{(0.0093)} & 0.612 \tiny{(0.0154)}\\
  1e5 & 0.895 \tiny{(0.0097)} &  0.910 \tiny{(0.0090)} & 0.602 \tiny{(0.0155)}\\ \hline

  \end{tabular}
\end{table}

One reason for the reduction in volume of the ellipsoid is that  multivariate methods capture information ignored by univariate analysis. This also leads to a better understanding of the effective samples obtained in an MCMC sample. \cite{vat:fle:jon:2015} provide the following estimator of effective sample size 
\[ 
 n \left(\dfrac{|\widehat{\Lambda}|}{|\widehat{\Sigma}|} \right)^{1/p},
\]
where $\widehat{\Lambda}$ is the sample covariance matrix for
$g(X_t)$, $\widehat{\Sigma}$ is a strongly consistent estimator of
$\Sigma$, and $|\cdot|$ denotes determinant. They demonstrate the
superiority of this estimator of effective sample size to the
univariate estimator of \cite{kass:carl:gel:neal:1998} and
\cite{gong:fleg:2015}. 


The rest of the paper is organized as follows. In Section
\ref{sec:multi_sve} we formally define the MSVE and present conditions
for strong consistency. We also establish strong consistency of the
eigenvalues. Section \ref{sec:simulation} contains a simulation study
where we investigate the finite sample properties of the MSVE in the
context of a vector autoregressive process. Finally, we present a
discussion in Section \ref{sec:discussion}. Many technical details of
the proofs from Section \ref{sec:multi_sve} are deferred to the
appendices.

\section{Spectral Estimators and Results} 
\label{sec:multi_sve}

\subsection{Definition of MSVE}
\label{sec:MSVE}

Let $Y_t = g(X_t) - \theta$, $t=1,2,3,\ldots$ and define the lag $s$,
$s \geq 0$, autocovariance matrix as
\[
\gamma(s) = \gamma(-s)^T = \E_F \left[Y_t \, Y_{t+s}^T \right]\, .
\]
Define $I_s$ as $I_s = \{1, \dots, (n-s)\}$ for $s \geq 0$ and as $I_s = \{(1-s), \dots, n\}$ for $s < 0$. Let $\bar{Y}_n = n^{-1} \sum_{t=1}^{n} Y_t$ and define the lag $s$
sample autocovariance  as
\begin{equation}
\label{eq:gamma_n}
\gamma_n(s) = \dfrac{1}{n} \sum_{t \in I_s}(Y_t - \bar{Y}_n)(Y_{t+s} -
\bar{Y}_n)^T\,.
\end{equation}
The MSVE is a weighted and truncated sum of the lag $s$ sample
autocovariances,
\begin{equation}
\label{eq:msve}
	\widehat{\Sigma}_S = \ds \sum_{s = -(b_n-1)}^{b_n-1} w_n(s) \gamma_n(s),
\end{equation}
where $w_n(\cdot)$ is the \textit{lag window} and $b_n$ is the \textit{truncation point}. 

\subsection{Strong Consistency of MSVE}
\label{sec:strong_cons}

\subsubsection{Strong Invariance Principle}
\label{sec:sip}

While Markov chains are our primary interest, we only require
$\{X_t\}_{t\geq 1}$ to be a stochastic process which satisfies a strong
invariance principle or SIP.  In the interest of clarity, the SIP was
stated somewhat loosely in Section~\ref{sec:introduction}. What
follows is a formal statement of our assumption.

Recall that $F$ is a distribution having support $\X$, $g: \X \to
\real^p$, and we are interested in estimating $\theta = \E_F g$.  We
assume $g^2$ (where the square is element-wise) is an $F$-integrable
function.  Set $h(X_t) = \left[ g(X_t) - \theta\right] ^2$, let $\|
\cdot \|$ denote the Euclidean norm, and let $B(t)$ denote a 
$p$-dimensional standard Brownian motion.

We will require an SIP for the partial sums of both $g$ and $h$.  We
assume there exists a $p \times p$ lower triangular matrix $L$, an increasing function
$\psi$ on the integers, a finite random variable $D$, and a
sufficiently rich probability space such that, with probability 1,
\begin{equation}
\label{eq:multi_sip}
\left\| \sum_{t=1}^{n} g(X_{t}) - n\theta - LB(n) 
\right\| < D\, \psi(n) \, .
\end{equation}
We also assume there exists a finite $p$-vector $\theta_h$, a $p \times
p$ lower triangular matrix $L_h$, an increasing function $\psi_h$ on
the integers, a finite random variable $D_h$, and a sufficiently rich
probability space such that, with probability 1,
\begin{equation}
\label{eq:multi_sip_h}
\left\| \sum_{t=1}^{n} h(X_{t}) - n\theta_h - L_hB(n) 
\right\| < D_h\, \psi_h(n) \, .
\end{equation}

\begin{rem} \label{rem:sip} Strong invariance principles have
  attracted much research interest and have been shown to hold for a
  wide variety of processes; see Section~\ref{sec:discussion} for some
  discussion on this point.  Results from \cite{kuel:phil:1980} show
  that for the Markov chains commonly encountered in MCMC settings,
  \eqref{eq:multi_sip} and \eqref{eq:multi_sip_h} hold with $\psi (n)
  = \psi_h (n) = n^{1/2 - \lambda}$ for some $\lambda > 0$.  The
  correlation of the process is measured indirectly by $\psi$
  \citep{phil:stou:1975}; a large serial correlation implies $\lambda$
  is closer to 0 while for less correlated processes $\lambda$ is
  closer to 1/2.
\end{rem}

\subsubsection{Strong Consistency}
\label{sec:sc}

In \eqref{eq:msve} we define the MSVE as the weighted and truncated
sum of the lag $s$ sample autocovariances.  We make the following
assumptions on the lag window $w_n(\cdot)$ and the truncation point
$b_n$.

\begin{cond}
\label{cond:lag.window}
The lag window $w_n(\cdot)$ is an even function defined on $\mathbb{Z}$ such that
\begin{enumerate}[(a)]
  \item $|w_n(s)| \leq 1$ for all $n$ and $s$, 
  \item $w_n(0) = 1$ for all $n$, and 
  \item $w_n(s) = 0$ for all $|s| \geq b_n$.
\end{enumerate}
\end{cond}
\citet{ande:1971} gives a list of lag windows that satisfy
Condition~\ref{cond:lag.window}.  We will consider some of these
further in Section~\ref{sec:lag window}. 

The following Conditions \ref{cond:bn} and \ref{cond:bn_c_logn} are
technical conditions ensuring that $b_n$ grows at the right rate
compared to $n$.
\begin{cond}
\label{cond:bn}
Let $b_n$ be an integer sequence such that $b_n \to \infty$ and $n/b_n \to \infty$ as $n \to \infty$ where $b_n$ and $n/b_n$ are non-decreasing. 
\end{cond}

\begin{cond}
\label{cond:bn_c_logn}
Let $b_n$ be an integer sequence such that
\begin{enumerate}[(a)]
  \item there exists a constant $c \geq 1$ such that $\sum_n (b_n/n)^c < \infty$,
  \item $b_nn^{-1} \log n \to 0$ as $n\to \infty$,
  \item $b_n^{-1} \log n = O(1)$, and
  \item $n > 2b_n$.
\end{enumerate}
\end{cond}

If $b_n = \lfloor n^{\nu} \rfloor$, where $ 0 < \nu < 1$, then Condition
\ref{cond:bn_c_logn} is satisfied if $n > 2^{1/(1-\nu)}$.

Define
\[
\Delta_1 w_n(k)  = w_n(k-1) - w_n(k)
\]
and
\[
\Delta_2 w_n(k)  = w_n(k-1) - 2w_n(k) + w_n(k+1) \; .
\]
\begin{cond}
\label{cond:bn_lag_psi}
Let $b_n$ be an integer sequence, $w_n$ be the lag window, and $\psi(n)$ and $\psi_h(n)$ be positive functions on the integers such that,
\begin{enumerate}[(a)]
  \item   
$b_n n^{-1} \sum_{k=1}^{b_n}k |\Delta_1 w_n(k)| \to 0$ as  $n\to \infty$, \label{cond:bn_lag}
  \item $b_n \psi(n)^2 \log n \left(\ds \sum_{k=1}^{b_n} |\Delta_2 w_n(k)|  \right)^2 \to 0$ as $n\to \infty$, \label{cond:wn_and_psi_1}
  \item $\psi(n)^2 \ds \sum_{k=1}^{b_n} |\Delta_2 w_n(k)| \to 0$ as $n\to \infty$, \label{cond:wn_and_psi_2}
  \item $b_n^{-1} \psi_h(n) \to 0$ as $n\to \infty$, and \label{cond:bn_psih}
  \item $b_n^{-1} \psi(n) \to 0$ as $n \to \infty$.\label{cond:bn_psi}
\end{enumerate}
\end{cond}

Condition \ref{cond:bn_lag_psi}\ref{cond:bn_lag} connects the
truncation point $b_n$ to the lag window $w_n$. In Section
\ref{sec:lag window} we will present examples of lag windows that
satisfy this condition. The functions $\psi(n)$ and $\psi_h(n)$ in
Conditions \ref{cond:bn_lag_psi}\ref{cond:wn_and_psi_1},
\ref{cond:bn_lag_psi}\ref{cond:wn_and_psi_2},
\ref{cond:bn_lag_psi}\ref{cond:bn_psih}, and
\ref{cond:bn_lag_psi}\ref{cond:bn_psi} correspond to the functions
described in \eqref{eq:multi_sip} and \eqref{eq:multi_sip_h} and thus
these four conditions connect the truncation point $b_n$, the lag
window $w_n$, and the correlation of the process, measured indirectly
by $\psi(n)$ and $\psi_h(n)$. In Lemma \ref{lemma:lag.window.bound} we
present sufficient conditions for Conditions
\ref{cond:bn_lag_psi}\ref{cond:bn_lag},
\ref{cond:bn_lag_psi}\ref{cond:wn_and_psi_1}, and
\ref{cond:bn_lag_psi}\ref{cond:wn_and_psi_2}.

\begin{thm}
\label{thm:main}
Suppose the strong invariance principles \eqref{eq:multi_sip} and
\eqref{eq:multi_sip_h} hold. If Conditions \ref{cond:lag.window},
\ref{cond:bn}, \ref{cond:bn_c_logn}, and \ref{cond:bn_lag_psi} hold,
then $\widehat{\Sigma}_S \to \Sigma$, with probability 1, as $n\to
\infty$.
\end{thm}

\begin{proof}[Outline of proof]
The proof is split into several lemmas; see Appendix~\ref{sec:appendixA}
for details.  Define for $l = 0, \dots, (n- b_n)$,
$\bar{Y}_{l}(k) = k^{-1} \sum_{t=1}^{k} Y_{l+t}$ and
\begin{equation*}
\widehat{\Sigma}_{w,n} = \dfrac{1}{n} \ds \sum_{l=0}^{n-b_n} \sum_{k=1}^{b_n} k^2 \Delta_2 w_n(k) [\bar{Y}_l(k) - \bar{Y}_n][\bar{Y}_l(k) - \bar{Y}_n]^T \, . 	
\end{equation*}
 For $t = 1, \dots, n$, define $Z_t = Y_t - \bar{Y}_n$. Then, in
 Lemma~\ref{lemma:dn} we show that $\widehat{\Sigma}_{w,n} =
 \widehat{\Sigma}_S - d_n$, where
\begin{align}
d_n & = \dfrac{1}{n}  \left\{ \sum_{t=1}^{b_n} \Delta_1w_n(t)  \left( \sum_{l=1}^{t-1} Z_lZ_l^T + \sum_{l= n-b_n+t+1}^{n} Z_lZ_l^T  \right) \right. \nonumber \\
&\quad \left. +    \sum_{s=1}^{b_n-1}\left[  \sum_{t=1}^{b_n-s} \Delta_1 w_n(s+t) \left( \sum_{l=1}^{t-1} \left(Z_lZ_{l+s}^T + Z_{l+s} Z_{l}^T \right) + \sum_{l=n-b_n+t+1}^{n-s} \left(Z_lZ_{l+s}^T + Z_{l+s} Z_{l}^T \right)  \right)  \right]\right\} \;. \label{eq:dn1}
\end{align}
Notice that in \eqref{eq:dn1} we use the convention that empty sums
are zero. In Lemma \ref{lemma:dn.0} we show that $d_n \to 0$ as $n \to
\infty$ with probability 1. Thus $\widehat{\Sigma}_{w,n} -
\widehat{\Sigma}_S \to 0$, with probability 1, as $n \to \infty$. In
Lemma~\ref{lemma:wn.to.sigma}, we show that $\widehat{\Sigma}_{w,n}
\to \Sigma$, with probability 1, as $n \to \infty$, and the result
follows.
\end{proof}

We use Theorem~\ref{thm:main} to give conditions for the strong
consistency of $\widehat{\Sigma}_{S}$ when the underlying stochastic
process is a Harris ergodic Markov chain having invariant distribution
$F$, but first we need a couple of definitions.  Recall that $F$ has
support $\X$ and $\mathcal{B}(\X)$ is a countably generated
$\sigma$-field.  For $n \in \N = \{1, 2, 3,\ldots \}$, let the
$n$-step Markov kernel associated with $X$ starting at $x \in \X$ be
$P^{n}(x,dy)$.  Then if $A \in \mathcal{B}(\X)$ and $r\in \{1, 2, 3, 
\ldots \}$, $P^{n}(x,A) = \Pr (X_{r+n} \in A | X_{r} =x) $. Let
$\|\cdot\|_{TV}$ denote the total variation norm.  The Markov chain is
\textit{polynomially ergodic of order $\xi$} where $\xi > 0$ if there
exists $M: \X \to \real^+$ with $E_{F} M < \infty$ such that
\begin{equation} \label{eq:tvn}
\Vert P^n (x, \cdot) - F(\cdot) \Vert_{TV} \leq M(x) n^{-\xi} \; .
\end{equation}
Notice that polynomial ergodicity is weaker than geometric or uniform
ergodicity; see \citet{meyn:twee:2009}.

\begin{rem}\label{rem:polynomial_erg}
Polynomial ergodicity is often proved by establishing the following drift condition. For a function $V:\X \to [1, \infty)$ there exists $d > 0, b < \infty,$ and $0 \leq \tau < 1$ such that for $x \in \X$
\[ 
E[V(X_{n+1}) | X_{n} = x] - V(x) \leq -d [ V(x) ]^{\tau} + bI(x \in C) \, ,
\]
where $C$ is a small set. In order to verify that $E_{F} M < \infty$,
it is sufficient to show that $E_{F} V < \infty$ by Theorem 14.3.7 in
\cite{meyn:twee:2009}.
\end{rem}

\begin{thm}
\label{cor:polymc}
Suppose $\E_F \|g\|^{4 + \delta} < \infty$ for some $\delta > 0$.  Let
$X$ be a polynomially ergodic Markov chain of order $\xi \geq (1 +
\epsilon) (1 + 2/\delta)$ for some $\epsilon > 0$. Then
\eqref{eq:multi_sip} and \eqref{eq:multi_sip_h} hold with 
\[
\psi(n) = \psi_h (n) = n^{1/2 - \lambda}
\]
for some $\lambda > 0$ that depends on $p$, $\epsilon$, and $\delta$.
If Conditions \ref{cond:lag.window}, \ref{cond:bn},
\ref{cond:bn_c_logn}, and \ref{cond:bn_lag_psi} hold, then
$\widehat{\Sigma}_S \to \Sigma$, with probability 1, as $n\to \infty$.
\end{thm}

\begin{proof}
See Appendix~\ref{sec:appendixB}.
\end{proof}

\begin{rem} \label{rem:lambda} We rely on results provided by
  \cite{kuel:phil:1980} to establish the existence of
  \eqref{eq:multi_sip} and \eqref{eq:multi_sip_h} in
  Theorem~\ref{cor:polymc}.  However, the precise relationship of
  $\lambda$ with $p$, $\epsilon$, and $\delta$ is not investigated in
  \cite{kuel:phil:1980} and remains an open problem.
\end{rem}

\begin{rem}
  When $p=1$, the MSV estimator reduces to the spectral variance
  estimator (SVE) considered by \citet{atch:2011}, \citet{dame:1991},
  and \citet{fleg:jone:2010}.  However, our result requires weaker
  conditions.  First notice that \citet{fleg:jone:2010} required
  weaker conditions than \citet{dame:1991}.  Thus we only need to
  compare Theorem~\ref{cor:polymc} to the results in \citet{atch:2011}
  and \citet{fleg:jone:2010}, both of whom required the Markov chains
  to be geometrically ergodic and to satisfy a one-step minorization
  condition.  Thus Theorem~\ref{cor:polymc} substantially weakens the
  conditions on the underlying Markov chain, while extending the
  results to the $p \ge 1$ setting.
\end{rem}

\subsubsection{Strong Consistency of Eigenvalues}
\label{sec:eigenvalue}

Having obtained a strongly consistent estimator of $\Sigma$, it is
natural to consider the eigenvalues of the estimator.

\begin{thm}
\label{thm:eigen_con}
Let $\widehat{\Sigma}$ be any strongly consistent estimator of $\Sigma$
and let $\lambda_1 \geq \lambda_2 \geq \dots \geq \lambda_p > 0$ be the eigenvalues of $\Sigma$. Let $\hat{\lambda}_1, \dots, \hat{\lambda}_p$ be the $p$ eigenvalues of $\widehat{\Sigma}$ such that
$\hat{\lambda}_1 \geq \hat{\lambda}_2 \geq \dots \geq \hat{\lambda}_p$, then
$\hat{\lambda}_k \to \lambda_k$, with probability 1, as $n \to \infty$ for all $1 \leq k \leq p$.
\end{thm}

\begin{proof}
Let $\| \cdot \|_F$ denote the Frobenius norm. By Weyl's inequality \citep{frank:2012}, for $\epsilon > 0$, if $\| \widehat{\Sigma} - \Sigma \|_F \leq \epsilon$, then  for all $1 \leq k \leq p$, $|\hat{\lambda}_k - \lambda_k| \leq \epsilon,$ which gives the desired result.
\end{proof}

\begin{rem}
\label{rem:eigen}
Theorem~\ref{thm:eigen_con} immediately implies that under the
conditions of either Theorem~\ref{thm:main} or Theorem~\ref{cor:polymc} the
sample eigenvalues of the MSVE are consistent for the population eigenvalues.
That is, $\hat{\lambda}_k \to \lambda_k$, with probability 1, as $n
\to \infty$ for all $1 \leq k \leq p$.
\end{rem}

Sample eigenvalues can play an important role in multivariate
analyses.  For example, the length of any axis of the confidence
region constructed from $\widehat{\Sigma}_{S}$ is determined by the
magnitude of the relevant estimated sample eigenvalue.  Thus the largest
eigenvalue is associated with the axis having the largest estimated
Monte Carlo error.  This also suggests that dimension reduction
methods could be useful in assessing the reliability of the simulation
effort.  Although this is a potentially interesting research direction
it is beyond the scope of this paper.

\subsubsection{Lag Window Conditions}
\label{sec:lag window}

The following generalization
of Lemma 7 in \cite{fleg:jone:2010} is useful for checking that a lag
window satisfies the conditions of Theorem~\ref{thm:main}.

\begin{lemma}
\label{lemma:lag.window.bound}
Reparameterize $w_n$ such that $w_n$ is defined on $[0,1]$ and $w_n(0) = 1$ and $w_n(1) = 0$. Further assume that $w_n$ is twice continuously differentiable and that there exists finite constants $D_1$ and $D_2$ such that $|w_n'(x)| \leq D_1$ and $|w_n''(x)| < D_2$. Then as $n \to \infty$,

\begin{enumerate}
\item Condition \ref{cond:bn_lag_psi}\ref{cond:bn_lag} holds if $b_n^2 n^{-1} \to 0$,

\item Conditions \ref{cond:bn_lag_psi}\ref{cond:wn_and_psi_1} and \ref{cond:bn_lag_psi}\ref{cond:wn_and_psi_2} holds if $b_n^{-1} \psi(n)^2 \log n \to 0$.
 \end{enumerate} 
\end{lemma}

\begin{proof}
The argument is the same as that of Lemma 7 in \cite{fleg:jone:2010}
and hence is omitted.
\end{proof}

\begin{rem}
  It is common to use $b_n = \lfloor n^{\nu} \rfloor$ in which case Conditions
  \ref{cond:bn_lag_psi}\ref{cond:bn_lag},
  \ref{cond:bn_lag_psi}\ref{cond:wn_and_psi_1}, and
  \ref{cond:bn_lag_psi}\ref{cond:wn_and_psi_2} hold, if we choose
  $0 < \nu < 1/2$ such that $n^{-\nu} \psi(n)^2 \log n \to 0$ as $n
  \to \infty$.
\end{rem}

\begin{rem}
We now consider some examples of lag windows which satisfy
Condition~\ref{cond:lag.window} and consider whether Conditions  \ref{cond:bn_lag_psi}\ref{cond:bn_lag}, \ref{cond:bn_lag_psi}\ref{cond:wn_and_psi_1}, and \ref{cond:bn_lag_psi}\ref{cond:wn_and_psi_2}  hold.

\begin{enumerate}
\item \textit{Simple Truncation}: $w_n(k) = I(|k| < b_n)$. Using this
  window the estimator obtained is truncated at $b_n$ but weighted
  identically. In this case, $\Delta_2 w_n(k) = 0$ for $k =
  1, \dots, b_n-2$, $\Delta_2 w_n(b_n-1) = -1$ and $\Delta_2 w_n(b_n)
  = 1$. It is easy to see that Condition \ref{cond:bn_lag_psi}\ref{cond:wn_and_psi_2}  is not satisfied.\\

\item \textit{Blackman-Tukey}: $w_n(k) = \left[1 - 2a + 2a \cos \left( \pi
      |k|/{b_n} \right)\right]I(|k| < b_n)$ where $a > 0$. This is a
  generalization for the \textit{Tukey-Hanning} window where $a =
  1/4$. For fixed $a$, the Blackman-Tukey window satisfies the
  conditions of Lemma \ref{lemma:lag.window.bound}, thus Conditions  \ref{cond:bn_lag_psi}\ref{cond:bn_lag}, \ref{cond:bn_lag_psi}\ref{cond:wn_and_psi_1}, and \ref{cond:bn_lag_psi}\ref{cond:wn_and_psi_2} hold if $b_n^2
  n^{-1} \to 0$ and $b_n^{-1} \psi(n)^2 \log n \to 0$ as $n \to
  \infty$.\\

\item \textit{Parzen}: $w_n(k) = \left[1 - |k|^q/{b_n^q} \right] I(|k| <
  b_n)$ for $q \in \mathbb{Z}^+$. When $q = 1$ this is the
  \textit{modified Bartlett} window. It is easy to show that the
  Parzen window satisfies the conditions for Lemma
  \ref{lemma:lag.window.bound}, and thus Conditions  \ref{cond:bn_lag_psi}\ref{cond:bn_lag}, \ref{cond:bn_lag_psi}\ref{cond:wn_and_psi_1}, and \ref{cond:bn_lag_psi}\ref{cond:wn_and_psi_2} hold if $b_n^2 n^{-1} \to
  0$ and $b_n^{-1} \psi(n)^2 \log n \to 0$ as $n \to \infty$.\\

\item \textit{Scale-parameter modified Bartlett}: $w_n(k) = \left[1 -
    \eta |k|/{b_n} \right] I(|k| < b_n)$ where $\eta$ is a positive
  constant not equal to 1. Then $\Delta_1 w_n(k) = \eta b_n^{-1}$ for
  $k = 1, 2, \dots, b_n-1$ and $\Delta_1 w_n(b_n) = 1 - \eta + \eta
  b_n^{-1}$ so that Condition \ref{cond:bn_lag_psi}\ref{cond:bn_lag}
  is satisfied when $b_n^2 n^{-1} \to 0$ as $n \to \infty$. Also,
  $\Delta_2 w_n(k) = 0$ for $k = 1, 2, \dots, b_n-2$, $\Delta_2
  w_n(b_n-1) = \eta - 1$ and $\Delta_2 w_n(b_n) = 1 - \eta + \eta
  b_n^{-1}$. We conclude that $\sum_{k=1}^{b_n} |\Delta_2 w_n(k)|$
  does not converge to 0 and hence Condition
  \ref{cond:bn_lag_psi}\ref{cond:wn_and_psi_2} is not satisfied.
\end{enumerate}
\end{rem}

Figure \ref{fig:windows} provides a graph of the three lag windows we
consider in the next section, specifically, the modified Bartlett,
Tukey-Hanning, and scale-parameter modified Bartlett windows. It is
evident that the modified Bartlett and Tukey-Hanning windows are
similar and the scale-parameter modified Bartlett window weighs the
lags more severely.
\begin{figure}
\centering
		\includegraphics[scale = .4]{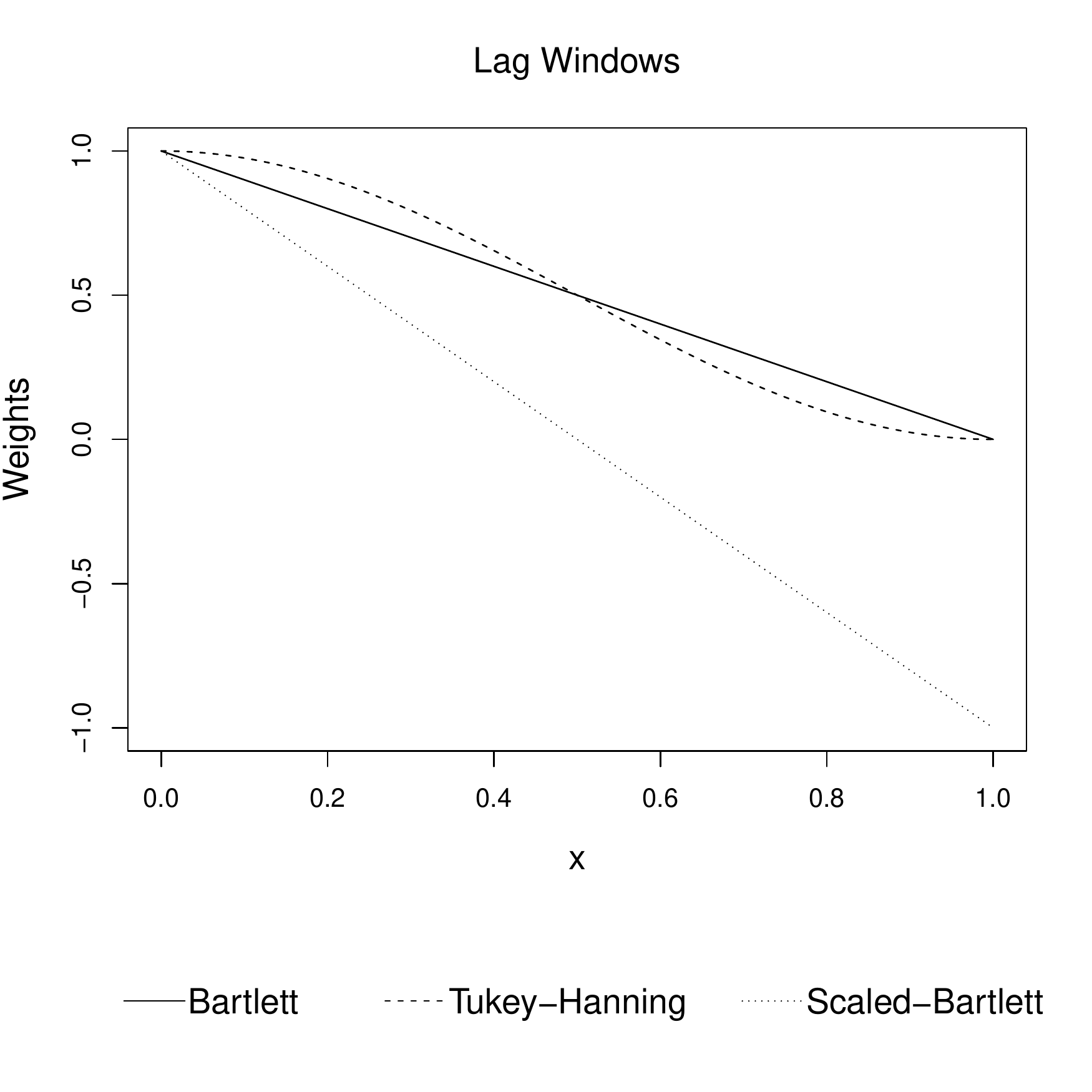}
                \caption{Plot of three lag windows, modified Bartlett
                  (Bartlett), Tukey-Hanning and the scale-parameter
                  Bartlett with scale parameter 2 (Scaled-Bartlett).}
	\label{fig:windows}
\end{figure}

\section{Simulation} 
\label{sec:simulation}

We consider some finite sample properties of the MSVE in the context of a
vector autoregressive process of order 1 or VAR(1).  Let
\begin{equation}
\label{eq:var}
	y_t = \Phi y_{t-1} + \epsilon_t,
\end{equation}
where $y_t \in \mathbb{R}^p$ for all $t$, $\Phi$ is a $p \times p$
matrix, $\epsilon_t \overset{iid}{\sim} N_p(0, W)$, and $y_0$ is the
zero vector. While this is a simple model, it is useful to study since
we can control the correlation of the process.

We assume that the largest eigenvalue of $\Phi$,
$\phi_{\max}$, is less than 1 in absolute value, in which case the stationary
distribution for the process is $ F = N_p(0, V)$ where $vec(V) =
(I_{p^2} - \Phi \otimes \Phi)^{-1} vec(W)$. Here $\otimes$
denotes Kronecker product and $I_{p^2}$ is the $p^2 \times p^2$ identity matrix. With some algebra it can be shown that
the lag $s$ autocovariance matrix for $s > 0$ is
\[ \gamma (s) = \Phi^s V  \, \, \, \, \text{and} \, \, \, \,  \gamma(-s) = V (\Phi^T)^s.\]

Consider estimating $\E_Fy$ with $\bar{y}_n$, the Monte Carlo
estimate. \cite{tjos:1990} showed that the process is geometrically
ergodic as long as $|\phi_{\max}| < 1$. In fact, the smaller the
largest eigenvalue, the faster the process mixes. Since $F$ has a
moment generating function, a CLT holds with
\begin{align*}
	\Sigma & =  \ds \sum_{s = -\infty}^{\infty} \gamma(s)\\
	& = \ds \sum_{s = 0}^{\infty} \gamma(s) + \ds \sum_{s = -\infty}^{0} \gamma(s) - V\\
	& = \ds \sum_{s = 0}^{\infty} \Phi^s V + \ds \sum_{s = -\infty}^{0} V (\Phi^T)^s - V\\
	& = (1 - \Phi)^{-1} V + V \left(1 - \Phi^T \right)^{-1} - V. \numberthis \label{eq:true_sig}
\end{align*}

For this process, we investigate the performance of the class of MSVE
in estimating $\Sigma$. We set $W$ to be the first order
autoregressive covariance matrix with correlation $\rho = 0.5$ and
present simulation results for different settings of $\Phi$ and $p$. These settings are presented in
Table \ref{tab:settings}. For Settings 1 and 4, $\phi_{\max} = .2$,
Settings 2 and 5, $\phi_{\max} = .6$ and Settings 3 and 6,
$\phi_{\max} = .9$. Thus, these three pairs of settings yield
processes with different mixing rates.

\begin{table}[!h]
    \caption{\footnotesize Simulation settings 1 through 6.}
  \label{tab:settings}
	\begin{center}
		\begin{tabular}{ccc}
		\hline

		\hline
		Setting & $p$ & Eigenvalues of $\Phi$ for $i = 0, \dots, p-1$\\
		\hline
		1 &	10 & $\lambda_i = .01 + i(.20 - .01)/(p-1)$ \\
		2 &	10 & $\lambda_i = .40 + i(.60 - .40)/(p-1)$ \\
		3 &	10 & $\lambda_i = .70 + i(.90 - .70)/(p-1)$ \\
		4 &	50 & $\lambda_i = .01 + i(.20 - .01)/(p-1)$ \\
		5 &	50 & $\lambda_i = .40 + i(.60 - .40)/(p-1)$ \\
		6 &	50 & $\lambda_i = .70 + i(.90 - .70)/(p-1)$ \\
		\hline

		\hline
		\end{tabular}
	\end{center}
\end{table}

We compare the performance of three lag windows: modified Bartlett,
Tukey-Hanning, and scale-parameter modified Bartlett with scale $= 2$.
In Section \ref{sec:multi_sve} we showed that the modified Bartlett
and the Tukey-Hanning windows satisfy the conditions of Theorem
\ref{thm:main} while the scale-parameter modified Bartlett does not.

For each setting, we do the following in each of 100 independent
replications. We observe the process for a Monte Carlo sample size of
$1e5$, and calculate the three MSVEs at samples $\{1e3, 5e3,
1e4, 5e4, 1e5\}$ with $b_n = \lfloor n^{1/3} \rfloor$. The error in
estimation is determined by calculating the average relative difference in
Frobenius norm, i.e. $||\widehat{\Sigma} - \Sigma||_F/||\Sigma||_F$
for each of the three windows at all five Monte Carlo sample sizes.

In Figure \ref{fig:p10_50_conv}, we plot the results for all settings
for all three lag windows. For Settings 1 and 4, all three lag windows
perform equally well while for Settings 3 and 6, the scale parameter
modified Bartlett window performs poorly. In all settings, the
modified Bartlett and the Tukey-Hanning windows perform similarly, but
the Tukey-Hanning window is slightly better when the chain mixes more
slowly. The plots also indicate that as $\phi_{\max}$ increases, a
larger Monte Carlo sample size is required for a desired error in
estimation threshold. This is as expected since we know for higher
values of $\phi_{\max}$, the process mixes more slowly.
\begin{figure}
	\begin{center}
		\includegraphics[width = 5in]{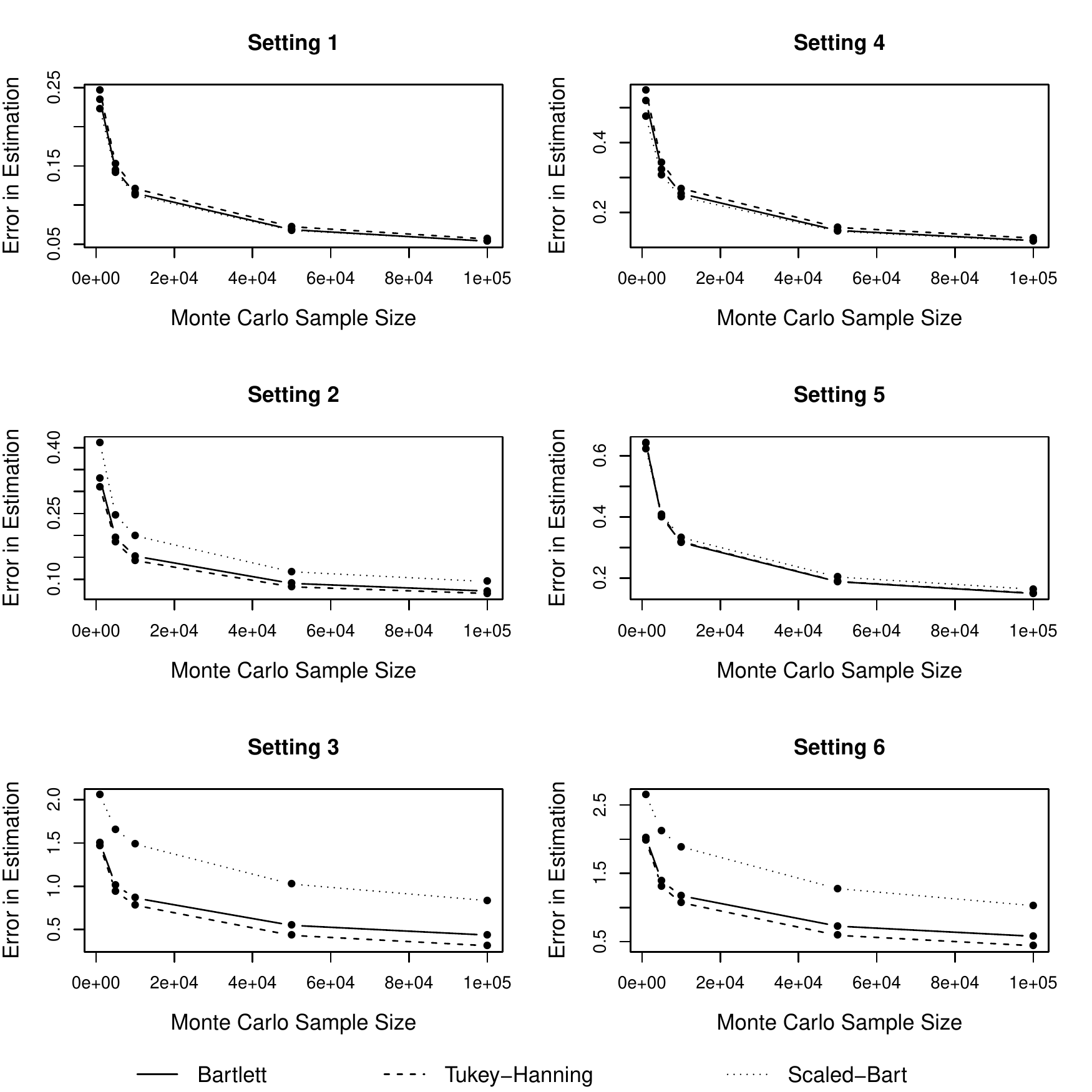}
	\end{center}
	\caption{$\|\widehat{\Sigma}_S - \Sigma\|_F/ \|\Sigma\|_F$ for the three lag windows at different Monte Carlo sample sizes for all six settings averaged over 100 iterations.}
	\label{fig:p10_50_conv}
\end{figure}

In Section \ref{sec:multi_sve} we presented the proof for the
convergence of the eigenvalues of the MSVE in Remark~\ref{rem:eigen}.  To study the finite
sample properties of the maximum eigenvalue we observe its behavior
for the three different lag windows at different Monte Carlo sample
sizes over each of 100 independent replications. At each replication,
we observe the relative error in estimation, $|\hat{\lambda}_1 -
\lambda_1|/\lambda_1$. The results are presented in Figure
\ref{fig:eig_p10_50} and are similar to what was observed for the
convergence of the MSVEs. For Settings 2, 3, 5 and 6, the
scale-parameter modified Bartlett window performs significantly worse
than the Tukey-Hanning and the modified Bartlett windows.  When
the chain mixes more slowly, the Tukey-Hanning window appears to give
slightly better results.

\begin{figure}
	\begin{center}
		\includegraphics[width = 5in]{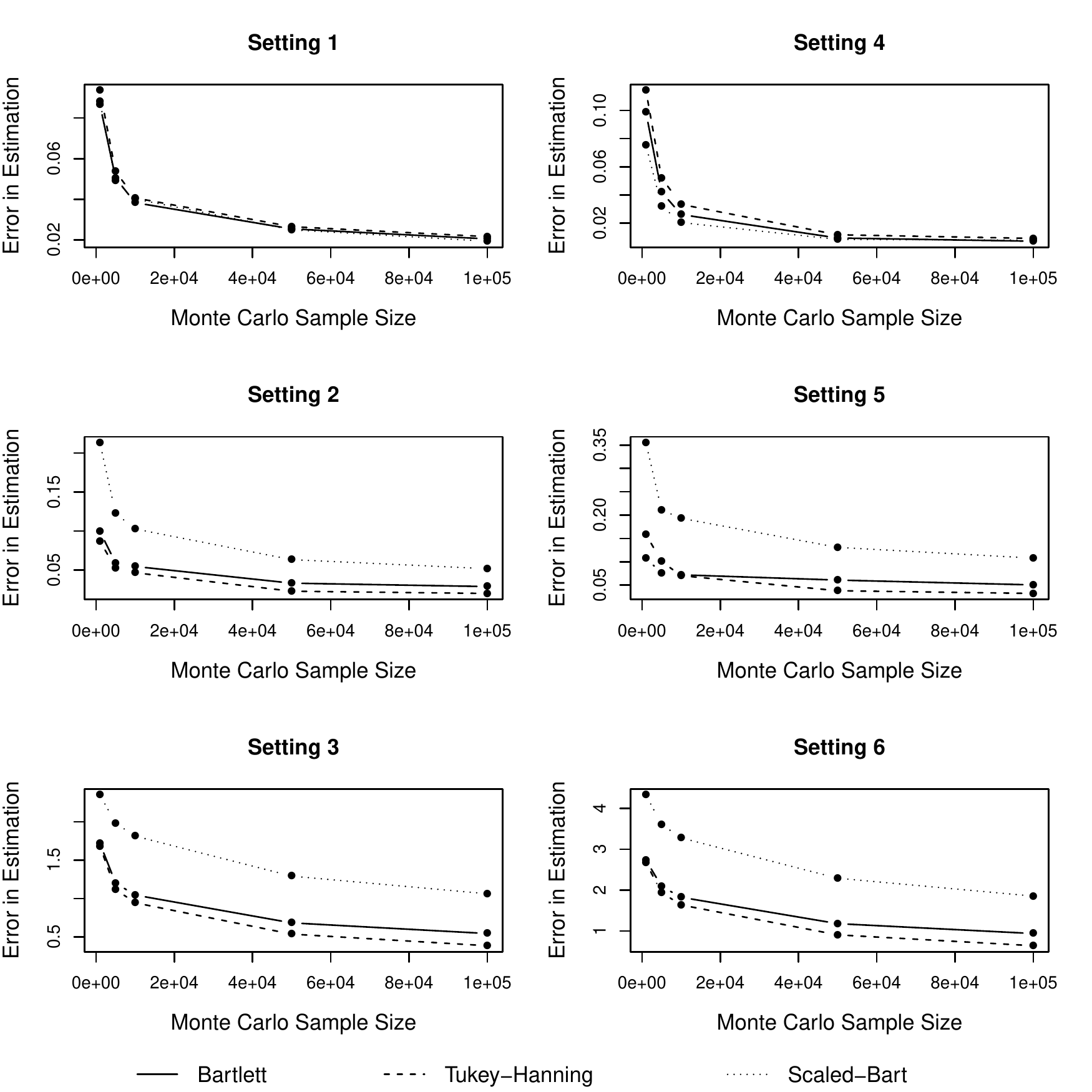}
	\end{center}
	\caption{$|\hat{\lambda}_1 - \lambda_1|/\lambda_1$ for the three lag windows at different Monte Carlo sample sizes for all six settings averaged over 100 iterations.}
	\label{fig:eig_p10_50}
\end{figure}

It is natural to investigate the stability of estimation of the
largest eigenvalue.  We study this empirically for Setting 1 by
observing the shape of the distribution of the maximum eigenvalue for
the estimates of $\Sigma$ obtained through the three lag windows at
varying Monte Carlo sample sizes over the 100 independent
replications. Using \eqref{eq:true_sig}, the true maximum eigenvalue
for this setting is 2.683. In Figure \ref{fig:eig_clt}, we notice that
as the Monte Carlo sample size increases, the shape of the density of
the largest eigenvalue is increasingly symmetric and centered at this
true value. In addition, as the Monte Carlo sample size increases, the
variance of the largest estimated eigenvalue decreases. This is observed for all
three lag windows.


\begin{figure}%
\centering
\subfloat[$n = 10^3$]{%
\label{fig:eig_p2_clt1}%
\includegraphics[width=.34\linewidth]{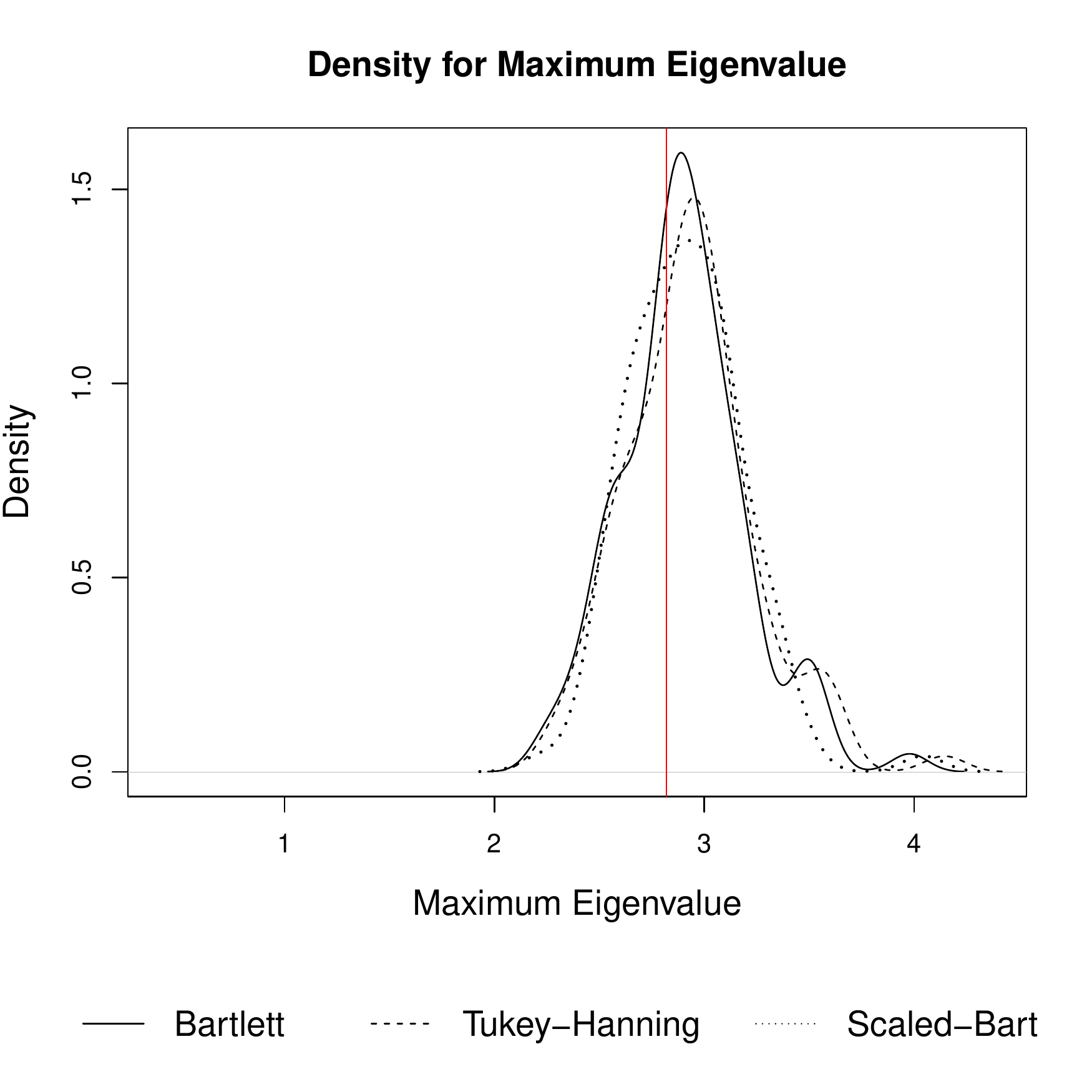}}%
\subfloat[$n = 10^4$]{%
\label{fig:eig_p2_clt2}%
\includegraphics[width=.34\linewidth]{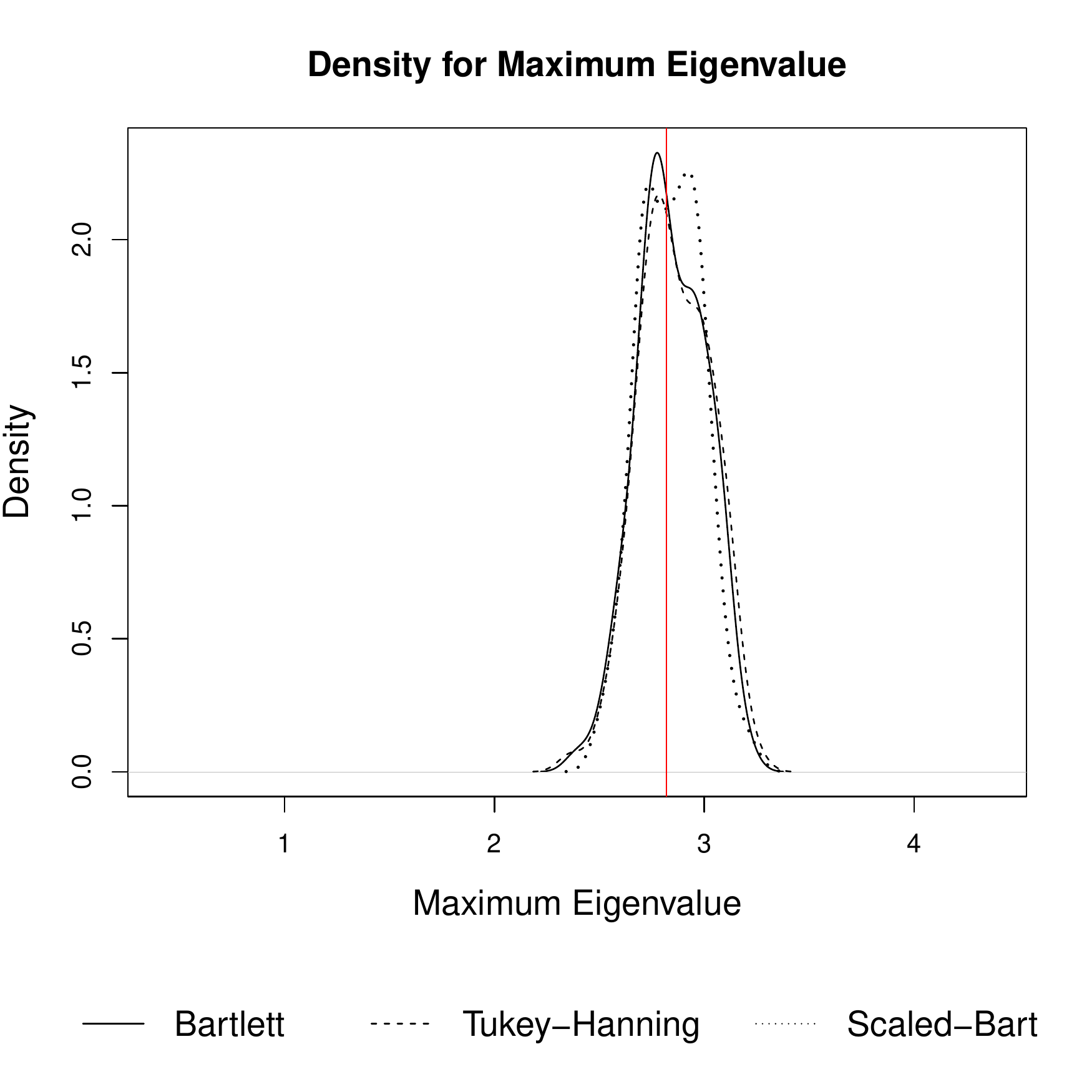}} 
\subfloat[$n = 10^5$]{%
 \label{fig:eig_p2_clt3}%
 \includegraphics[width=.34\linewidth]{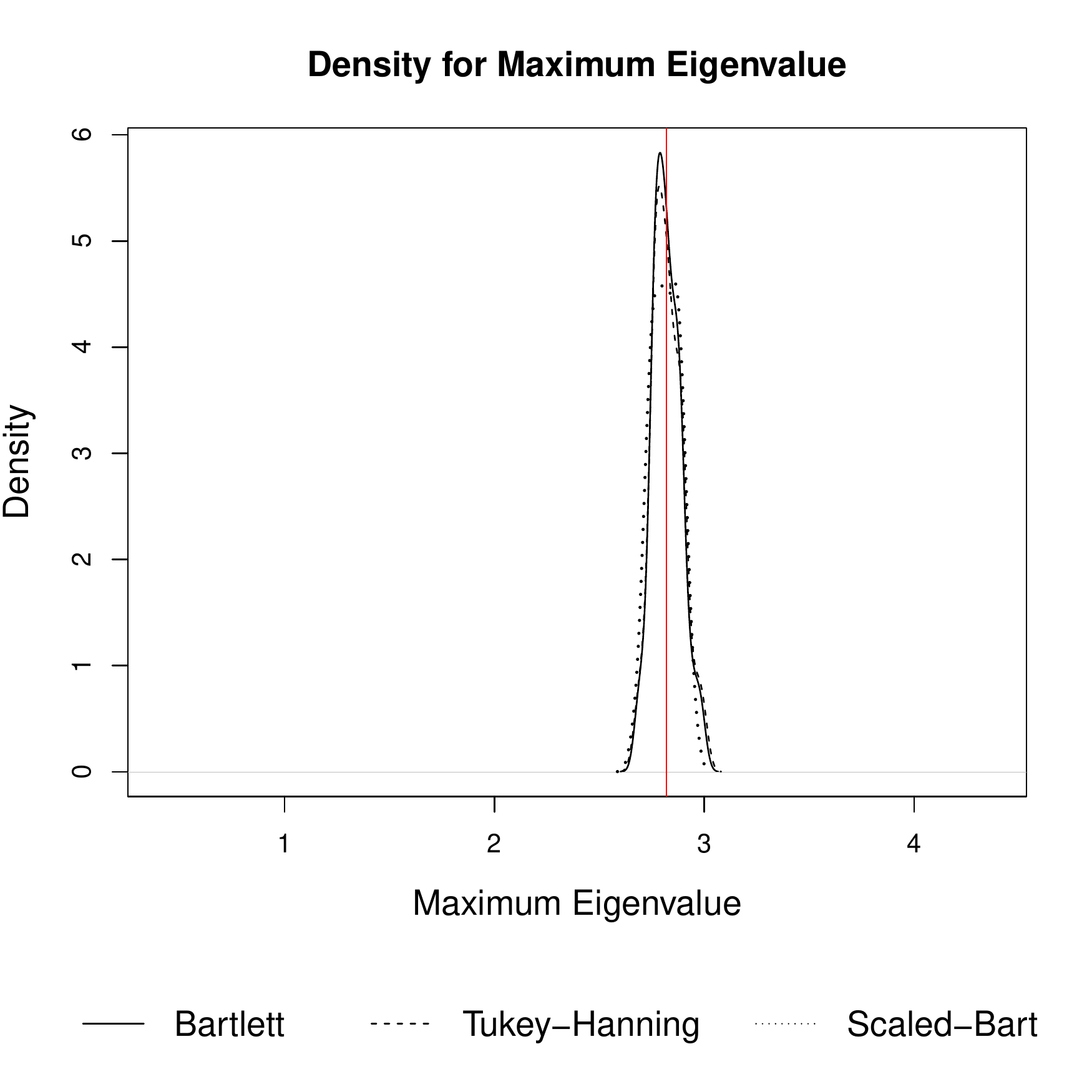}}
 \caption{Kernel density of the maximum eigenvalue for the three MSVEs over 100 replications and increasing Monte Carlo sample sizes for Setting 1. The vertical line indicates the true eigenvalue of 2.683 calculated using \eqref{eq:true_sig}.}%
 \label{fig:eig_clt}
\end{figure}


\FloatBarrier

\section{Discussion} 
\label{sec:discussion}

Estimation of the asymptotic covariance matrix in the CLT as in
\eqref{eq:multi_clt} has received little attention in the MCMC
literature thus far.  Due to the results of this paper, practitioners
are now equipped with a class of strongly consistent multivariate
spectral variance estimators of $\Sigma$.

However, multivariate spectral variance estimators are also
encountered outside of the MCMC context. For example, they are often
used for heteroscedastic and autocorrelation consistent (HAC) estimation of covariance matrices
which, for example, arise in the study of generalized method of
moments and autoregressive processes with heteroscedastic errors. See
\cite{andr:1991} for motivating examples. In the context of HAC
estimation, \cite{dejo:2000} obtained conditions under which the class
of MSVEs are strongly consistent.  However, these conditions are
restrictive in the context of MCMC.  In particular, his Assumption 2
\citep[][page 264]{dejo:2000} will not be satisfied in many typical
MCMC applications.  Additionally, we require weaker mixing conditions
on the underlying stochastic process. That is, although Markov chains
are the primary focus for us, our results hold for much more general
stochastic processes as we explain below.

Our main assumption on the underlying stochastic process are the SIPs
as stated in \eqref{eq:multi_sip} and \eqref{eq:multi_sip_h}.  The
existence of an SIP has attracted much research interest.  Consider
the univariate case.  For independent and identically distributed
(i.i.d) processes, the first result of this kind is due to
\cite{stra:1964} who showed $\psi(n) = \sqrt{n \log \log
  n}$. \cite{koml:majo:tusn:1975} found that if $\E_F |g|^{2 + \delta}
< \infty$, then $\psi(n) = n^{1/2 - \lambda}$ for $\lambda > 0$ (often
called the KMT bound). \cite{koml:majo:tusn:1975} also showed that if
$g$ has all moments in a neighborhood of 0, then $\psi(n) = \log
n$. The results of \cite{koml:majo:tusn:1975} are the strongest to
date in the i.i.d setting. The main reference for a univariate strong
invariance principle for dependent sequences is \cite{phil:stou:1975}
who prove bounds similar to that of \cite{koml:majo:tusn:1975} for a
variety of weakly dependent processes including $\phi$-mixing,
regenerative and strongly mixing processes.  Also, see \citet{wu:2007}
for a univariate strong invariance principle for certain classes of
dependent processes.

Many of the univariate SIPs have been extended to the multivariate
setting. For independent processes, \cite{berk:phil:1979},
\cite{einm:1989}, and \cite{zait:1998} extend the results of
\cite{koml:majo:tusn:1975}. For correlated processes, \cite{eber:1986}
showed the existence of a strong invariance principle for Martingale
sequences and \cite{horv:1984} proved the KMT bound for multivariate
extended renewal processes. For $\phi$-mixing, strongly mixing, and
absolutely regular processes, \cite{kuel:phil:1980} and
\cite{dehl:phil:1982} extended the \cite{phil:stou:1975} results to
the multivariate case. 

\section{Acknowledgment} 

The authors thank Tiefeng Jiang and Gongjun Xu for helpful discussions.

\begin{appendix}

\section{Strong Consistency of MSVE} 
\label{sec:appendixA}

Before we begin the proof of Theorem~\ref{thm:main} we note some
useful properties of Brownian motion and lag
windows which will be used often throughout the proof.

\subsection{Brownian Motion}
\label{sec:BM}

Recall that $\{B(t)\}_{t\geq 0}$ denotes a $p$-dimensional standard
Brownian motion and that $B^{(i)}$ denotes the $i$th component of $B(t)$.
 
\begin{lemma}[\citet{csor:reve:1981}]
\label{bn.diff.bound}
Suppose Condition \ref{cond:bn} holds, then for all $\epsilon > 0$ and for almost all sample paths, there exists $n_0(\epsilon)$ such that for all $n \geq n_0$ and all $i = 1, \dots, p$
\[\sup_{0\leq t \leq n-b_n} \sup_{0 \leq s \leq b_n} \left|B^{(i)}(t+s) - B^{(i)}(t)\right| < (1+\epsilon) \left(2b_n \left(\log{\dfrac{n}{b_n}} + \log{\log n}  \right) \right) ^{1/2},  \]
\[\sup_{0 \leq s \leq b_n} \left|B^{(i)}(n) - B^{(i)}(n-s) \right| < (1+\epsilon) \left(2b_n \left(\log{\dfrac{n}{b_n}} + \log{\log n}  \right) \right) ^{1/2} \text{, and } \] 
\[\left| B^{(i)}(n) \right|  < (1 + \epsilon) \sqrt{2n \log \log n}.\]
\end{lemma}

Let $L$ be a lower triangular matrix and set $\Sigma=LL^{T}$. Define
$C(t) := LB(t)$ and if $C^{(i)}(t)$ is the $i$th component of $C(t)$, define
\[
\bar{C}^{(i)}_l(k)  := \dfrac{1}{k} \left(C^{(i)}(l+k) - C^{(i)}(l)
\right) \text{ and } \bar{C}^{(i)}_n  := \dfrac{1}{n} C^{(i)}(n)
\; .
\]
Since $C^{(i)}(t) \sim N(0, t \Sigma_{ii})$, where $\Sigma_{ii}$ is the $i$th diagonal of $\Sigma$, $C^{(i)}/\sqrt{\Sigma_{ii}}$ is a 1-dimensional standard Brownian motion. As a consequence, we have the following corollaries of Lemma \ref{bn.diff.bound}.

\begin{corollary}
Suppose Condition \ref{cond:bn} holds, then for all $\epsilon > 0$ and for almost all sample paths there exists $n_0(\epsilon)$ such that for all $n \geq n_0$ and all $i = 1, \dots, p$
\begin{equation}
\label{n.bound}
\left|C^{(i)} (n)\right| < (1 + \epsilon) (2n \Sigma_{ii} \log{\log n})^{1/2},
\end{equation}
where $\Sigma_{ii}$ is the $i$th diagonal entry of $\Sigma$. 
\end{corollary}

\begin{corollary}
Suppose Condition \ref{cond:bn} holds, then for all $\epsilon > 0$ and for almost all sample paths, there exists $n_0(\epsilon)$ such that for all $n \geq n_0$ and all $i = 1, \dots, p$
\begin{equation}
\label{diff.bound}
\left|\bar{C}^{(i)}_l(k)\right|  \leq \dfrac{1}{k}  \sup_{0 \leq l \leq n-b_n} \sup_{0 \leq s \leq b_n} \left|C^{(i)}(l+s) - C^{(i)}(l) \right|  <   \dfrac{1}{k} 2 (1 + \epsilon) (b_n \Sigma_{ii} \log{n})^{1/2},
\end{equation}
where $\Sigma_{ii}$ is the $i$th diagonal entry of $\Sigma$. 
\end{corollary}

\subsection{Basic Properties of Lag Windows}

Recall that the lag window $w_n(\cdot)$ is such that it satisfies Condition \ref{cond:lag.window}. We will require the following results about the lag window $w_n(\cdot)$.
\begin{lemma}[\citet{dame:1991}]
\label{lemma:lag.res}
Under Condition \ref{cond:lag.window}, 
\begin{enumerate}[(i)]
\item $\Delta_1 w_n(s) = \ds \sum_{k=s}^{b_n} \Delta_2w_n(k)$, 
\item $ \ds \sum_{k=s+1}^{b_n} \Delta_1 w_n(k) = w_n(s)$, and
\item $\ds \sum_{k=1}^{b_n} \Delta_1 w_n(k) = 1$.
\end{enumerate}
\end{lemma}

\subsection{Proof of Theorem~\ref{thm:main}}

Recall that
\begin{equation}
\label{eq:sigma_wn}
\widehat{\Sigma}_{w,n} = \dfrac{1}{n} \ds \sum_{l=0}^{n-b_n} \sum_{k=1}^{b_n} k^2 \Delta_2 w_n(k) [\bar{Y}_l(k) - \bar{Y}_n][\bar{Y}_l(k) - \bar{Y}_n]^T \, .   
\end{equation}
For $t = 1, 2, \dots, n$, define $Z_t = Y_t - \bar{Y}_n$ and
\begin{align}
d_n  & = \dfrac{1}{n}  \left\{ \sum_{t=1}^{b_n} \Delta_1w_n(t)  \left( \sum_{l=1}^{t-1} Z_lZ_l^T + \sum_{l= n-b_n+t+1}^{n} Z_lZ_l^T  \right) \right. \nonumber \\
&\quad \left. +    \sum_{s=1}^{b_n-1}\left[  \sum_{t=1}^{b_n-s} \Delta_1 w_n(s+t) \left( \sum_{l=1}^{t-1} \left(Z_lZ_{l+s}^T + Z_{l+s} Z_{l}^T \right) + \sum_{l=n-b_n+t+1}^{n-s} \left(Z_lZ_{l+s}^T + Z_{l+s} Z_{l}^T \right)  \right)  \right]\right\} \;. \label{eq:dn}
\end{align}
Notice that in \eqref{eq:dn} we use the convention that empty sums are zero. 
\begin{lemma}
\label{lemma:dn}

Under Condition \ref{cond:lag.window}, $\widehat{\Sigma}_{w,n} = \widehat{\Sigma}_{S} - d_{n}$.
\end{lemma}

\begin{proof}
For  $i,j = 1, \dots, p$, let $\widehat{\Sigma}_{w,ij}$ denote the $(i,j)$th entry of $\widehat{\Sigma}_{w,n}$. Then,
\begin{align*}
\widehat{\Sigma}_{w,ij} & = \dfrac{1}{n} \ds \sum_{l=0}^{n-b_n} \ds\sum_{k=1}^{b_n} k^2 \Delta_2 w_n(k) \left[\bar{Y}^{(i)}_l(k) - \bar{Y}^{(i)}_n  \right] \left[\bar{Y}^{(j)}_l(k) - \bar{Y}^{(j)}_n  \right]\\
& = \dfrac{1}{n} \ds \sum_{l=0}^{n-b_n} \ds\sum_{k=1}^{b_n}  \Delta_2 w_n(k) \left[\ds \sum_{t = 1}^{k} Y^{(i)}_{l+t} - k\bar{Y}^{(i)}_n  \right]\left[\ds \sum_{t = 1}^{k} Y^{(j)}_{l+t} - k\bar{Y}^{(j)}_n  \right]\\
& = \dfrac{1}{n} \ds \sum_{l=0}^{n-b_n} \ds\sum_{k=1}^{b_n}  \Delta_2 w_n(k) \left[\ds \sum_{t=1}^{k}Z^{(i)}_{l+t}  \right]\left[\ds \sum_{t=1}^{k}Z^{(j)}_{l+t}  \right]\\
& = \dfrac{1}{n} \ds \sum_{l=0}^{n-b_n} \ds\sum_{k=1}^{b_n}  \Delta_2 w_n(k) \left[\ds \sum_{t=1}^{k} Z^{(i)}_{l+t}Z^{(j)}_{l+t} + \ds \sum_{s=1}^{k-1} \ds \sum_{t=1}^{k-s}Z^{(i)}_{l+t}Z^{(j)}_{l+t+s} + \ds \sum_{s=1}^{k-1} \ds \sum_{t=1}^{k-s}Z^{(j)}_{l+t}Z^{(i)}_{l+t+s}  \right]. \numberthis \label{eq:zterms}
\end{align*}
Notice that in \eqref{eq:zterms}, we use the convention that empty sums are zero. 
We will consider each term in \eqref{eq:zterms} separately. For the first term, changing the order of summation  and then using Lemma \ref{lemma:lag.res},
\begin{align*}
& \dfrac{1}{n} \ds \sum_{l=0}^{n-b_n} \ds\sum_{k=1}^{b_n} \ds \sum_{t=1}^{k} \Delta_2 w_n(k) Z^{(i)}_{l+t}Z^{(j)}_{l+t}\\
 & = \dfrac{1}{n} \ds \sum_{l=0}^{n-b_n} \ds\sum_{t=1}^{b_n} \ds \sum_{k=t}^{b_n} \Delta_2 w_n(k) Z^{(i)}_{l+t}Z^{(j)}_{l+t}\\
& = \dfrac{1}{n} \ds \sum_{l=0}^{n-b_n} \ds\sum_{t=1}^{b_n} \Delta_1 w_n(t) Z^{(i)}_{l+t}Z^{(j)}_{l+t} \\
& = \dfrac{1}{n} \ds \sum_{t=1}^{b_n}  \Delta_1 w_n(t) \ds \sum_{l = 0}^{n-b_n}   Z^{(i)}_{l+t}Z^{(j)}_{l+t}\\
& =  \ds \sum_{t=1}^{b_n}  \Delta_1 w_n(t) \left[  \gamma_{n,ij}(0) - \dfrac{1}{n}  \left( Z^{(i)}_1 Z^{(j)}_1 + \dots +Z^{(i)}_{t-1}Z^{(j)}_{t-1} + Z^{(i)}_{n-b_n+t+1}Z^{(j)}_{n-b_n+t+1}  +  \dots + Z^{(i)}_{n}Z^{(j)}_{n}   \right) \right]\\
& =  \gamma_{n,ij}(0)  - \dfrac{1}{n}  \ds \sum_{t=1}^{b_n} \Delta_1 w_n(t) \left(\ds \sum_{l=1}^{t-1} Z^{(i)}_{l}Z^{(j)}_{l}  + \ds \sum_{l=n-b_n+t+1}^{n} Z^{(i)}_{l} Z^{(j)}_{l}  \right)   \quad \text{by Lemma \ref{lemma:lag.res}}. \numberthis \label{eq:zterms1}
\end{align*}
For the second term in \eqref{eq:zterms} we change the order of summation from $l,k,s,t$ to
$l,s,k,t$ to $l,s,t,k$ to get
 \begin{align*}
& \frac{1}{n} \sum_{l=0}^{n-b_n} \sum_{k=1}^{b_n}  \sum_{s=1}^{k-1}
 \sum_{t=1}^{k-s} \Delta_2 w_n(k) Z^{(i)}_{l+t}Z^{(j)}_{l+t+s} \\
 &= \dfrac{1}{n} \ds \sum_{l=0}^{n-b_n} \ds\sum_{s=1}^{b_n-1}   \ds
 \sum_{k=s+1}^{b_n} \ds \sum_{t=1}^{k-s} \Delta_2 w_n(k)
 Z^{(i)}_{l+t}Z^{(j)}_{l+t+s}  \\
& = \dfrac{1}{n} \ds \sum_{l=0}^{n-b_n} \ds\sum_{s=1}^{b_n-1}   \ds
\sum_{t=1}^{b_n-s} \ds \sum_{k = t+s}^{b_n} \Delta_2 w_n(k)
Z^{(i)}_{l+t}Z^{(j)}_{l+t+s} \\
 & = \dfrac{1}{n} \ds \sum_{l=0}^{n-b_n} \ds\sum_{s=1}^{b_n-1}   \ds
 \sum_{t=1}^{b_n-s}  \Delta_1 w_n(s+t) Z^{(i)}_{l+t}Z^{(j)}_{l+t+s}
 \quad \quad \text{by Lemma \ref{lemma:lag.res}}\\
 & = \dfrac{1}{n} \ds \sum_{s=1}^{b_n-1} \ds\sum_{l=0}^{n - b_n}   \ds \sum_{t=1}^{b_n-s}  \Delta_1 w_n(s+t) Z^{(i)}_{l+t}Z^{(j)}_{l+t+s}\\
  & = \dfrac{1}{n} \ds \sum_{s=1}^{b_n-1} \ds\sum_{t=1}^{b_n-s}   \ds \sum_{l=0}^{n-b_n}  \Delta_1 w_n(s+t) Z^{(i)}_{l+t}Z^{(j)}_{l+t+s}\\
  & = \dfrac{1}{n} \ds \sum_{s=1}^{b_n-1} \ds\sum_{t=1}^{b_n-s} \Delta_1 w_n(s+t)  \ds \sum_{l=0}^{n-b_n}Z^{(i)}_{l+t}Z^{(j)}_{l+t+s}\\
    & = \ds \sum_{s=1}^{b_n-1} \ds\sum_{t=1}^{b_n-s} \Delta_1 w_n(s+t) \left[ \gamma_{n,ij}(s) -  \dfrac{1}{n} \ds \sum_{l=1}^{t-1}Z^{(i)}_{l}Z^{(j)}_{l+s} - \dfrac{1}{n}\ds \sum_{l=n-b_n+t+1}^{n-s}Z^{(i)}_{l}Z^{(j)}_{l+s} \right]\\
& = \ds \sum_{s = 1}^{b_n-1} w_n(s) \gamma_{n,ij}(s)\\
& \quad   -  \dfrac{1}{n}
\ds \sum_{s=1}^{b_n-1} \ds\sum_{t=1}^{b_n-s} \Delta_1 w_n(s+t)\left[ \ds \sum_{l=1}^{t-1} Z^{(i)}_{l} Z^{(j)}_{l+s} +
  \ds \sum_{l=n-b_n+t+1}^{n-s} Z^{(i)}_{l} Z^{(j)}_{l+s} \right] \quad
\text{by Lemma \ref{lemma:lag.res}}. \numberthis \label{eq:zterms2}
 \end{align*}

Repeating the same steps as in the second term we reduce the third term in \eqref{eq:zterms} to 
\begin{align*}
& \dfrac{1}{n} \ds \sum_{l=0}^{n-b_n} \ds\sum_{k=1}^{b_n}  
 \sum_{s=1}^{k-1} \ds \sum_{t=1}^{k-s} \Delta_2 w_n(k)
 Z^{(j)}_{l+t}Z^{(i)}_{l+t+s} \\ 
& =  \ds \sum_{s = 1}^{b_n-1} w_n(s)
 \gamma_{n,ji}(s) -  \dfrac{1}{n} \ds \sum_{s=1}^{b_n-1} \ds\sum_{t=1}^{b_n-s} \Delta_1 w_n(s+t)  \left[ \ds \sum_{l=1}^{t-1} Z^{(j)}_{l} Z^{(i)}_{l+s} + \ds \sum_{l=n-b_n+t+1}^{n-s} Z^{(j)}_{l} Z^{(i)}_{l+s} \right]\\
& = \ds \sum_{s = 1}^{b_n-1} w_n(-s) \gamma_{n,ij}(-s)  -  \dfrac{1}{n} \ds \sum_{s=1}^{b_n-1} \ds\sum_{t=1}^{b_n-s}
\Delta_1 w_n(s+t)  \left[ \ds \sum_{l=1}^{t-1} Z^{(j)}_{l}
  Z^{(i)}_{l+s} + \ds \sum_{l=n-b_n+t+1}^{n-s} Z^{(j)}_{l}
  Z^{(i)}_{l+s} \right] \; .\numberthis \label{eq:zterms3}
\end{align*}
Using \eqref{eq:zterms1}, \eqref{eq:zterms2}, and \eqref{eq:zterms3}
in \eqref{eq:zterms}
\begin{align*}
 \widehat{\Sigma}_{w,ij} = & \, \gamma_{n,ij}(0) +  \ds \sum_{s = 1}^{b_n-1} w_n(s) \gamma_{n,ij}(s) + \ds \sum_{s = -(b_n-1)}^{-1}  w_n(s) \gamma_{n,ij}(s)\\
&  - \dfrac{1}{n}  \ds \sum_{t=1}^{b_n} \Delta_1 w_n(t) \left(\ds \sum_{l=1}^{t-1} Z^{(i)}_{l}Z^{(j)}_{l}  + \ds \sum_{l=n-b_n+t+1}^{n} Z^{(i)}_{l}Z^{(j)}_{l}  \right) \\
& -  \dfrac{1}{n} \ds \sum_{s=1}^{b_n-1} \ds\sum_{t=1}^{b_n-s} \Delta_1 w_n(s+t)  \left[ \ds \sum_{l=1}^{t-1} (Z^{(i)}_{l}Z^{(j)}_{l+s} + Z^{(i)}_{l+s} Z^{(j)}_l) + \ds \sum_{l=n-b_n+t+1}^{n-s}(Z^{(i)}_{l} Z^{(j)}_{l+s} + Z^{(i)}_{l+s} Z^{(j)}_l) \right]\\
= & \ds \sum_{s = -(b_n-1)}^{b_n-1}  \gamma_{n,ij}(s) w_n(s) - d_{n,ij}\\
= & \widehat{\Sigma}_{S, ij} - d_{n,ij}.
\end{align*}
\end{proof}

Let $\widetilde{\gamma}_n(s)$, $\widetilde{\Sigma}_S,
\widetilde{\Sigma}_{w,n}$ and $\tilde{d}_n$ be the Brownian motion
analogs of \eqref{eq:gamma_n}, \eqref{eq:msve}, \eqref{eq:sigma_wn},
and \eqref{eq:dn}. Specifically, for $t = 1, \dots, n$, define
Brownian motion increments $U_t= B(t) - B(t-1)$, so that $U_1, \dots,
U_n$ are $\overset{iid}{\sim}$ N$_p(0, I_p$) where $I_p$ is the $p \times p$ identity matrix. For $l = 0, \dots,
n-b_n$ and $k = 1, \dots, b_n$ define $\bar{B}_l(k) = k^{-1} (B(l+k) -
B(l))$,  $\bar{B}_n = n^{-1} B(n)$, and $T_t = U_t -
\bar{B}_n$. Then 
\begin{align}
& \widetilde{\gamma}_n(s)  = \dfrac{1}{n} \ds \sum_{t \in I_s} (U_t - \bar{B}_n)(U_{t+s} - \bar{B}_n)^T = \dfrac{1}{n}\ds \sum_{t \in I_s} T_tT_{t+s}^T, \label{eq:gammatilde}\\
& \widetilde{\Sigma}_S = \ds \sum_{s = -(b_n-1)}^{b_n-1} w_n(s) \widetilde{\gamma}_n(s), \label{eq:sigmatilde_S}\\
& \widetilde{\Sigma}_{w,n} = \dfrac{1}{n} \ds \sum_{l=0}^{n-b_n} \sum_{k=1}^{b_n} k^2 \Delta_2 w_n(k) [\bar{B}_l(k) - \bar{B}_n][\bar{B}_l(k) - \bar{B}_n]^T, \label{eq:sigmatilde_wn}\\
& \tilde{d}_n = \dfrac{1}{n} \left \{\ds \sum_{t=1}^{b_n} \Delta_1w_n(t)  \left(\ds \sum_{l=1}^{t-1} T_lT_l^T + \ds\sum_{l= n-b_n+t+1}^{n} T_lT_l^T  \right)\right. \notag  \\
& \quad + \left.\ds \sum_{s=1}^{b_n-1} \left[ \ds \sum_{t=1}^{b_n-s} \Delta_1 w_n(s+t) \left(\ds \sum_{l=1}^{t-1}  (T_lT_{l+s}^T + T_{l+s}T_{l}^T) + \sum_{l=n-b_n+t+1}^{n-s} (T_lT_{l+s}^T + T_{l+s}T_{l}^T) \right)  \right]  \right\}. \numberthis \label{eq:dtilde_n}
\end{align}
Notice that in \eqref{eq:dtilde_n} we use the convention that empty sums are zero. 
Our goal is to show that $\widetilde{\Sigma}_{w,n} \to I_p$ as $n \to \infty$ with probability 1 in the following way.  
In Lemma~\ref{lemma:dn.tilde} we show that $\widetilde{\Sigma}_{w,n} =
\widetilde{\Sigma}_S - \tilde{d}_n$ and in Lemma~\ref{lemma:dtilde.0}
we show that the end term $\tilde{d}_n \to 0$ as $n \to \infty$ with
probability 1. Lemma~\ref{lemma:sigma.tilde_wn_I} shows that
$\widetilde{\Sigma}_{S} \to I_p$ as $n \to \infty$ with probability 1,
and hence $\widetilde{\Sigma}_{w,n} \to I_p$ as $n \to \infty$ with
probability 1. 
\begin{lemma}
\label{lemma:dn.tilde}
Under Condition \ref{cond:lag.window}, $\widetilde{\Sigma}_{w,n} =
\widetilde{\Sigma}_{S} - \tilde{d}_{n}$.
\end{lemma}

\begin{proof}
For  $i,j = 1, \dots, p$, let $\widetilde{\Sigma}_{w,ij}$ denote the $(i,j)$th entry of $\widetilde{\Sigma}_{w,n}$. Then,
\begin{align*}
\widetilde{\Sigma}_{w,ij} & = \dfrac{1}{n} \ds \sum_{l=0}^{n-b_n} \ds\sum_{k=1}^{b_n} k^2 \Delta_2 w_n(k) \left[\bar{B}^{(i)}_l(k) - \bar{B}^{(i)}_n  \right] \left[\bar{B}^{(j)}_l(k) - \bar{B}^{(j)}_n  \right]\\
& = \dfrac{1}{n} \ds \sum_{l=0}^{n-b_n} \ds\sum_{k=1}^{b_n}  \Delta_2 w_n(k) \left[B^{(i)}(k+l) - B^{(i)}(l) - k\bar{B}^{(i)}_n  \right]\left[B^{(j)}(k+l) - B^{(j)}(l) - k\bar{B}^{(j)}_n \right]\\
& = \dfrac{1}{n} \ds \sum_{l=0}^{n-b_n} \ds\sum_{k=1}^{b_n}  \Delta_2 w_n(k) \left[ \ds \sum_{t=1}^{k}U^{(i)}_{t+l} -  k\bar{B}^{(i)}_n   \right]\left[ \ds \sum_{t=1}^{k}U^{(j)}_{t+l} -  k\bar{B}^{(j)}_n   \right]\\
& = \dfrac{1}{n} \ds \sum_{l=0}^{n-b_n} \ds\sum_{k=1}^{b_n}  \Delta_2 w_n(k) \left[\ds \sum_{t=1}^{k}T^{(i)}_{t+l}  \right]\left[\ds \sum_{t=1}^{k}T^{(j)}_{t+l}  \right]\\
& = \dfrac{1}{n} \ds \sum_{l=0}^{n-b_n} \ds\sum_{k=1}^{b_n}  \Delta_2 w_n(k) \left[\ds \sum_{t=1}^{k} T^{(i)}_{l+t}T^{(j)}_{l+t} + \ds \sum_{s=1}^{k-1} \ds \sum_{t=1}^{k-s}T^{(i)}_{l+t}T^{(j)}_{l+t+s} + \ds \sum_{s=1}^{k-1} \ds \sum_{t=1}^{k-s}T^{(j)}_{l+t}T^{(i)}_{l+t+s}  \right]. \numberthis \label{eq:tterms}
\end{align*}
In \eqref{eq:tterms}, we continue to use convention that empty sums are zero.
We will look at each of the terms in \eqref{eq:tterms} separately. 
For the first term, changing the order of summation
and then using Lemma \ref{lemma:lag.res},
 \begin{align*}
& \dfrac{1}{n} \ds \sum_{l=0}^{n-b_n} \ds\sum_{k=1}^{b_n} \ds
 \sum_{t=1}^{k} \Delta_2 w_n(k) T^{(i)}_{l+t}T^{(j)}_{l+t}\\
 & = \dfrac{1}{n} \ds \sum_{l=0}^{n-b_n} \ds\sum_{t=1}^{b_n} \ds \sum_{k=t}^{b_n} \Delta_2 w_n(k) T^{(i)}_{l+t}T^{(j)}_{l+t}\\
& = \dfrac{1}{n} \ds \sum_{l=0}^{n-b_n} \ds\sum_{t=1}^{b_n} T^{(i)}_{l+t}T^{(j)}_{l+t} \Delta_1 w_n(t)\\
& = \dfrac{1}{n} \ds \sum_{t=1}^{b_n}  \Delta_1 w_n(t) \ds \sum_{l = 0}^{n-b_n}   T^{(i)}_{l+t}T^{(j)}_{l+t}\\
& =  \widetilde{\gamma}_{n,ij}(0)  - \dfrac{1}{n}  \ds
\sum_{t=1}^{b_n} \Delta_1 w_n(t) \left(\ds \sum_{l=1}^{t-1}
  T^{(i)}_{l}T^{(j)}_{l}  + \ds \sum_{l=n-b_n+t+1}^{n}
  T^{(i)}_{l}T^{(j)}_{l}  \right).  \numberthis \label{tterms1}
\end{align*}
For the second term in \eqref{eq:tterms} we change the order of summation from $l,k,s,t$ to
$l,s,k,t$ then to $l,s,t,k$ and use  Lemma \ref{lemma:lag.res} to get
 \begin{align*}
& \frac{1}{n} \sum_{l=0}^{n-b_n} \sum_{k=1}^{b_n} \sum_{s=1}^{k-1}
 \sum_{t=1}^{k-s} \Delta_2 w_n(k) T^{(i)}_{l+t}T^{(j)}_{l+t+s}\\
&= \dfrac{1}{n} \ds \sum_{l=0}^{n-b_n} \ds\sum_{s=1}^{b_n-1}   \ds \sum_{k=s+1}^{b_n} \ds \sum_{t=1}^{k-s} \Delta_2 w_n(k) T^{(i)}_{l+t}T^{(j)}_{l+t+s}\\
 & = \dfrac{1}{n} \ds \sum_{l=0}^{n-b_n} \ds\sum_{s=1}^{b_n-1}   \ds \sum_{t=1}^{b_n-s} \ds \sum_{k = t+s}^{b_n} \Delta_2 w_n(k) T^{(i)}_{l+t}T^{(j)}_{l+t+s}\\
 & = \dfrac{1}{n} \ds \sum_{l=0}^{n-b_n} \ds\sum_{s=1}^{b_n-1}   \ds \sum_{t=1}^{b_n-s}  \Delta_1 w_n(s+t) T^{(i)}_{l+t}T^{(j)}_{l+t+s}\\
 & = \dfrac{1}{n} \ds \sum_{s=1}^{b_n-1} \ds\sum_{l=0}^{n - b_n}   \ds \sum_{t=1}^{b_n-s}  \Delta_1 w_n(s+t) T^{(i)}_{l+t}T^{(j)}_{l+t+s}\\
  & = \dfrac{1}{n} \ds \sum_{s=1}^{b_n-1} \ds\sum_{t=1}^{b_n-s}   \ds \sum_{l=0}^{n-b_n}  \Delta_1 w_n(s+t) T^{(i)}_{l+t}T^{(j)}_{l+t+s}\\
  & = \dfrac{1}{n} \ds \sum_{s=1}^{b_n-1} \ds\sum_{t=1}^{b_n-s} \Delta_1 w_n(s+t)  \ds \sum_{l=0}^{n-b_n}T^{(i)}_{l+t}T^{(j)}_{l+t+s}\\
      & = \dfrac{1}{n} \ds \sum_{s=1}^{b_n-1} \ds\sum_{t=1}^{b_n-s}
      \Delta_1 w_n(s+t)
      \ds\sum_{l=t}^{n-b_n+t}T^{(i)}_{l}T^{(j)}_{l+s}\\
    & = \ds \sum_{s=1}^{b_n-1} \ds\sum_{t=1}^{b_n-s} \Delta_1 w_n(s+t)
     \left[ \tilde{\gamma}_{n,ij}(s) -  \dfrac{1}{n} \ds
  \sum_{l=1}^{t-1}T^{(i)}_{l}T^{(j)}_{l+s} - \dfrac{1}{n}\ds
  \sum_{l=n-b_n+t+1}^{n-s}T^{(i)}_{l}T^{(j)}_{l+s} \right] \\
& = \ds \sum_{s = 1}^{b_n-1} w_n(s) \tilde{\gamma}_{n,ij}(s)  -
\dfrac{1}{n} \ds \sum_{s=1}^{b_n-1} \ds\sum_{t=1}^{b_n-s} \Delta_1
w_n(s+t)   \left[ \ds \sum_{l=1}^{t-1}T^{(i)}_{l}T^{(j)}_{l+s} + \ds \sum_{l=n-b_n+t+1}^{n-s}T^{(i)}_{l}T^{(j)}_{l+s} \right]. \numberthis \label{eq:tterms2}
 \end{align*}

Repeating the same steps as in the second term we reduce the third term in \eqref{eq:tterms} to 
\begin{align*}
& \dfrac{1}{n} \ds \sum_{l=0}^{n-b_n} \ds\sum_{k=1}^{b_n}   \ds \sum_{s=1}^{k-1} \ds \sum_{t=1}^{k-s} \Delta_2 w_n(k) T^{(j)}_{l+t}T^{(i)}_{l+t+s} \\
 &= \ds \sum_{s = 1}^{b_n-1} w_n(s) \tilde{\gamma}_{n,ji}(s)  -  \dfrac{1}{n} \ds \sum_{s=1}^{b_n-1} \ds\sum_{t=1}^{b_n-s} \Delta_1 w_n(s+t)  \left[ \ds \sum_{l=1}^{t-1}T^{(j)}_{l}T^{(i)}_{l+s} + \ds \sum_{l=n-b_n+t+1}^{n-s}T^{(j)}_{l}T^{(i)}_{l+s} \right]\\
& = \ds \sum_{s = 1}^{b_n-1} w_n(-s) \tilde{\gamma}_{n,ij}(-s)  -  \dfrac{1}{n} \ds \sum_{s=1}^{b_n-1} \ds\sum_{t=1}^{b_n-s} \Delta_1 w_n(s+t)  \left[ \ds \sum_{l=1}^{t-1}T^{(j)}_{l}T^{(i)}_{l+s} + \ds \sum_{l=n-b_n+t+1}^{n-s}T^{(j)}_{l}T^{(i)}_{l+s} \right]. \numberthis \label{eq:tterms3}
\end{align*}
Using \eqref{tterms1}, \eqref{eq:tterms2}, and \eqref{eq:tterms3} in \eqref{eq:tterms}, we get
\begin{align*}
 \widetilde{\Sigma}_{w,ij} = & \, \tilde{\gamma}_{n,ij}(0) +  \ds \sum_{s = 1}^{b_n-1} w_n(s) \tilde{\gamma}_{n,ij}(s) + \ds \sum_{s = -(b_n-1)}^{-1}  w_n(s) \tilde{\gamma}_{n,ij}(s)\\
&  - \dfrac{1}{n}  \ds \sum_{t=1}^{b_n} \Delta_1 w_n(t) \left(\ds \sum_{l=1}^{t-1} T^{(i)}_{l}T^{(j)}_{l}  + \ds \sum_{l=n-b_n+t+1}^{n} T^{(i)}_{l}T^{(j)}_{l}  \right) \\
& -  \dfrac{1}{n} \ds \sum_{s=1}^{b_n-1} \ds\sum_{t=1}^{b_n-s} \Delta_1 w_n(s+t)  \left[ \ds \sum_{l=1}^{t-1} (T^{(i)}_{l}T^{(j)}_{l+s} + T^{(i)}_{l+s} T^{(j)}_l) + \ds \sum_{l=n-b_n+t+1}^{n-s}(T^{(i)}_{l}T^{(j)}_{l+s} + T^{(i)}_{l+s} T^{(j)}_l) \right]\\
= & \ds \sum_{s = -(b_n-1)}^{b_n-1} \tilde{\gamma}_{n,ij}(s) w_n(s) - \tilde{d}_{n,ij}\\
= & \widetilde{\Sigma}_{S, ij} - \tilde{d}_{n,ij}.
\end{align*}
\end{proof}
Next, we show that
as $n \to \infty$, $\tilde{d}_n \to 0$ with probability 1 implying
$\widetilde{\Sigma}_{w,n} - \widetilde{\Sigma}_S \to 0$ with probability 1 as $n \to \infty$. To do so we
require a strong invariance principle for
independent and identically distributed random variables.

\begin{thm}[\citet{koml:majo:tusn:1975}]
  \label{thm:kmt}
   Let $B(n)$ be a 1-dimensional
        standard Brownian motion. If $X_1, X_2, X_3 \dots$ are
        independent and identically distributed univariate random
        variables with mean $\mu$ and standard deviation $\sigma$,
        such that $\E\left[e^{|tX_1|}\right] < \infty$ in a neighborhood of $t =
        0$, then as $n \to \infty$
  \begin{equation*}
     \sum_{i=1}^{n} X_i - n\mu - \sigma B(n) = O(\log
                n) \; .
   \end{equation*} 
\end{thm}

We begin with a technical lemma that will be used in a couple of places in the rest of the proof.
\begin{lemma}
\label{lem:window} 
Let Conditions \ref{cond:lag.window} and \ref{cond:bn} hold. If, as $n \to
\infty$, $b_n n^{-1} \sum_{k=1}^{b_n} k|\Delta_1 w_n(k)| \to 0$, then
\[
\dfrac{b_n}{n} \left(  \sum_{t=1}^{b_n} |\Delta_1 w_n(t)| + 2\ds \sum_{s=1}^{b_n-1} \ds \sum_{t=1}^{b_n-s} |\Delta_1 w_n(s+t)| \right)\to 0 \; .
\]

\end{lemma}

\begin{proof}
\begin{align*}
& \dfrac{b_n}{n} \left(  \sum_{t=1}^{b_n} |\Delta_1 w_n(t)| + 2\ds \sum_{s=1}^{b_n-1} \ds \sum_{t=1}^{b_n-s} |\Delta_1 w_n(s+t)| \right) \\
& =  \dfrac{b_n}{n} \left(  \sum_{t=1}^{b_n} |\Delta_1 w_n(t)| + 2\ds \sum_{s=1}^{b_n-1} \ds \sum_{k=s+1}^{b_n} |\Delta_1 w_n(k)| \right)\\
& =  \dfrac{b_n}{n} \left(  \sum_{t=1}^{b_n} |\Delta_1 w_n(t)| + 2\ds \sum_{k=2}^{b_n} \ds \sum_{s=1}^{k-1} |\Delta_1 w_n(k)| \right)\\
& = \dfrac{b_n}{n} \left(  \sum_{t=1}^{b_n} |\Delta_1 w_n(t)| + 2\ds \sum_{k=2}^{b_n} (k-1) |\Delta_1 w_n(k)| \right)\\
& = \dfrac{b_n}{n} \left(  \sum_{t=1}^{b_n} |\Delta_1 w_n(t)| + 2\ds \sum_{k=1}^{b_n} (k-1) |\Delta_1 w_n(k)| \right)\\
& \leq \dfrac{b_n}{n} \left( 2\ds \sum_{k=1}^{b_n} k |\Delta_1 w_n(k)| \right)\\
& \to 0 \text{ by assumption}.
\end{align*}
\end{proof}

\begin{lemma}
\label{lemma:dtilde.0}
Let Conditions \ref{cond:lag.window} and \ref{cond:bn} hold and let $n > 2b_n$.  If $b_n n^{-1} \sum_{k=1}^{b_n} k|\Delta_1 w_n(k)| \to 0$  and  $b_n^{-1} \log n = O(1)$ as $n \to \infty$, then  $\tilde{d}_n \to 0$ with probability 1 as $n \to \infty$.
\end{lemma}

\begin{proof}
For $i, j = 1, \dots, p$, we will show that as $n\to \infty$ with probability 1, $\tilde{d}_{n,ij} \to 0$. Recall
\begin{align*}
\tilde{d}_{n,ij} & = \dfrac{1}{n}  \ds \sum_{t=1}^{b_n} \Delta_1 w_n(t) \left(\ds \sum_{l=1}^{t-1} T^{(i)}_{l}T^{(j)}_{l}  + \ds \sum_{l=n-b_n+t+1}^{n} T^{(i)}_{l}T^{(j)}_{l}  \right) \\
& \quad +  \dfrac{1}{n} \ds \sum_{s=1}^{b_n-1} \ds\sum_{t=1}^{b_n-s} \Delta_1 w_n(s+t)  \left[ \ds \sum_{l=1}^{t-1} (T^{(i)}_{l}T^{(j)}_{l+s} + T^{(i)}_{l+s} T^{(j)}_l) + \ds \sum_{l=n-b_n+t+1}^{n-s}(T^{(i)}_{l}T^{(j)}_{l+s} + T^{(i)}_{l+s} T^{(j)}_l) \right]\\
|\tilde{d}_{n,ij}| & \leq \dfrac{1}{n}  \ds \sum_{t=1}^{b_n} |\Delta_1 w_n(t)| \left(\ds \sum_{l=1}^{t-1} \left|T^{(i)}_{l}T^{(j)}_{l}\right|  + \ds \sum_{l=n-b_n+t+1}^{n} \left|T^{(i)}_{l}T^{(j)}_{l}\right|  \right) \\
& \quad +  \dfrac{1}{n} \ds \sum_{s=1}^{b_n-1} \ds\sum_{t=1}^{b_n-s} |\Delta_1 w_n(s+t)| \, \,\times\\
&  \quad \times \left[ \ds \sum_{l=1}^{t-1} \left(\left|T^{(i)}_{l}T^{(j)}_{l+s}\right| + \left|T^{(i)}_{l+s} T^{(j)}_l \right| \right) + \ds \sum_{l=n-b_n+t+1}^{n-s} \left( \left| T^{(i)}_{l}T^{(j)}_{l+s} \right| + \left|T^{(i)}_{l+s} T^{(j)}_l \right| \right) \right], \numberthis \label{eq:abs.dn.tilde}
\end{align*}
where we use the convention that empty sums are zero. Using the inequality $|ab| \leq (a^2 + b^2)/2$ in the first and second terms in \eqref{eq:abs.dn.tilde}, we have for $t = 1, \dots, b_n$ 
\[ \ds \sum_{l=1}^{t-1} \left| T^{(i)}_{l}T^{(j)}_{l} \right| \leq \dfrac{1}{2}\ds \sum_{l=1}^{t-1} \left(T^{(i)2}_{l} +  T^{(j)2}_{l}  \right) \leq \dfrac{1}{2}\ds \sum_{l=1}^{2b_n} T^{(i)2}_{l}  + \dfrac{1}{2}\ds \sum_{l=1}^{2b_n} T^{(j)2}_{l}\]
\[ \ds \sum_{l=n-b_n+t+1}^{n} \left| T^{(i)}_{l}T^{(j)}_{l} \right| \leq \dfrac{1}{2}\ds \sum_{l=n-b_n+t+1}^{n} \left(T^{(i)2}_{l} +  T^{(j)2}_{l}  \right) \leq \dfrac{1}{2}\ds \sum_{l=n-2b_n+1}^{n} T^{(i)2}_{l}  + \dfrac{1}{2}\ds \sum_{l=n-2b_n+1}^{n} T^{(j)2}_{l}.\]

Similarly, for the third and fourth terms in \eqref{eq:abs.dn.tilde}, for $t = 1, \dots b_n - 1$ and $s = 1, \dots, b_n - 1$
\begin{align*}
 \ds \sum_{l=1}^{t-1} \left| T^{(i)}_{l}T^{(j)}_{l+s} \right| +  \ds \sum_{l=1}^{t-1} \left| T^{(i)}_{l+s}T^{(j)}_{l} \right|&  \leq \dfrac{1}{2}\ds \sum_{l=1}^{2b_n} T^{(i)2}_{l} + \dfrac{1}{2}\ds \sum_{l=1}^{2b_n} T^{(j)2}_{l} + \dfrac{1}{2}\ds \sum_{l=1}^{b_n} T^{(j)2}_{l+s} + \dfrac{1}{2}\ds \sum_{l=1}^{b_n} T^{(i)2}_{l+s}\\
 & \leq \dfrac{1}{2}\ds \sum_{l=1}^{2b_n} T^{(i)2}_{l} + \dfrac{1}{2}\ds \sum_{l=1}^{2b_n} T^{(j)2}_{l} + \dfrac{1}{2}\ds \sum_{l=1}^{2b_n} T^{(j)2}_{l} + \dfrac{1}{2}\ds \sum_{l=1}^{2b_n} T^{(i)2}_{l}
\\
& = \sum_{l=1}^{2b_n} T^{(i)2}_{l} + \ds \sum_{l=1}^{2b_n} T^{(j)2}_{l}.
\end{align*}
\begin{align*}
& \sum_{l=n-b_n+t+1}^{n} \left| T^{(i)}_{l}T^{(j)}_{l+s} \right| + \sum_{l=n-b_n+t+1}^{n} \left| T^{(i)}_{l+s}T^{(j)}_{l} \right| \\
&\leq  \dfrac{1}{2}\ds \sum_{l=n-2b_n+1}^{n} (T^{(i)2}_{l} +  T^{(j)2}_{l})  + \dfrac{1}{2}\ds \sum_{l=n-b_n+1}^{n} (T^{(j)2}_{l+s} + T^{(i)2}_{l+s})\\
& \leq \dfrac{1}{2}\ds \sum_{l=n-2b_n+1}^{n} (T^{(i)2}_{l} +  T^{(j)2}_{l})  + \dfrac{1}{2}\ds \sum_{l=n-2b_n+1}^{n} (T^{(j)2}_{l} + T^{(i)2}_{l})\\
& = \ds \sum_{l=n-2b_n+1}^{n} T^{(i)2}_{l} + \ds \sum_{l=n-2b_n+1}^{n} T^{(j)2}_{l}.
 \end{align*}
Combining the above results in \eqref{eq:abs.dn.tilde} we get, 
\begin{align*}
|\tilde{d}_{n,ij}|  & \leq \dfrac{1}{n}  \left(\dfrac{1}{2}\ds \sum_{l=1}^{2b_n} T^{(i)2}_{l}  + \dfrac{1}{2}\ds \sum_{l=1}^{2b_n} T^{(j)2}_{l} +  \dfrac{1}{2}\ds \sum_{l=n-2b_n+1}^{n} T^{(i)2}_{l}  + \dfrac{1}{2}\ds \sum_{l=n-2b_n+1}^{n} T^{(j)2}_{l}\right) \ds \sum_{t=1}^{b_n} |\Delta_1 w_n(t)| \\
& + \dfrac{1}{n} \ds \sum_{s=1}^{b_n-1} \left[ \left(  \ds \sum_{l=1}^{2b_n} T^{(i)2}_{l}  + \ds \sum_{l=1}^{2b_n} T^{(j)2}_{l} + \ds \sum_{l=n-2b_n+1}^{n} T^{(i)2}_{l}  + \ds \sum_{l=n-2b_n+1}^{n} T^{(j)2}_{l} \right) \ds\sum_{t=1}^{b_n-s} |\Delta_1 w_n(s+t)| \right]\\
& = \dfrac{1}{b_n} \left(  \dfrac{1}{2}\ds \sum_{l=1}^{2b_n} T^{(i)2}_{l}  + \dfrac{1}{2}\ds \sum_{l=1}^{2b_n} T^{(j)2}_{l} + \dfrac{1}{2}\ds \sum_{l=n-2b_n+1}^{n} T^{(i)2}_{l}  + \dfrac{1}{2}\ds \sum_{l=n-2b_n+1}^{n} T^{(j)2}_{l} \right) \times\\
 & \dfrac{b_n}{n} \left(  \sum_{t=1}^{b_n} |\Delta_1 w_n(t)| + 2\ds \sum_{s=1}^{b_n-1} \ds \sum_{t=1}^{b_n-s} |\Delta_1 w_n(s+t)| \right). \numberthis \label{eq:bounded_t}
\end{align*}

We will show that the first term in the product in \eqref{eq:bounded_t} remains bounded with probability 1 as $n\to \infty$. Consider,
\begin{align*}
  \dfrac{1}{2b_n} \ds \sum_{l=1}^{2b_n} T_l^{(i)2} & = \dfrac{1}{2b_n} \ds \sum_{l=1}^{2b_n} \left(U_l^{(i)}  - \bar{B}^{(i)}_n \right)^2\\
  & = \dfrac{1}{2b_n} \ds \sum_{l=1}^{2b_n} U_l^{(i)2} - 2 \bar{B}^{(i)}_n \dfrac{1}{2b_n} \ds \sum_{l=1}^{2b_n} U_l^{(i)} + \left(\bar{B}^{(i)}_n \right)^2\\
  & \leq \left|\dfrac{1}{2b_n} \ds \sum_{l=1}^{2b_n} U_l^{(i)2}  \right| + 2 \left|\bar{B}_n^{(i)}\right| \left|\dfrac{1}{2b_n} \ds \sum_{l=1}^{2b_n} U_l^{(i)}  \right| + \left| \bar{B}_n^{(i)}\right|^2 \\
  & < \left|\dfrac{1}{2b_n} \ds \sum_{l=1}^{2b_n} U_l^{(i)2}  \right| + \left|\dfrac{1}{2b_n} \ds \sum_{l=1}^{2b_n} U_l^{(i)}  \right| \left(\dfrac{2}{n} (1+\epsilon) (2n\log \log n)^{1/2} \right)\\
& \quad  + \left(\dfrac{1}{n} (1+\epsilon) (2n \log \log n)^{1/2}  \right)^2 \quad \quad \text{by Lemma \ref{bn.diff.bound}} \\
  & < \left|\dfrac{1}{2b_n} \ds \sum_{l=1}^{2b_n} U_l^{(i)2}  \right| + \left|\dfrac{1}{2b_n} \ds \sum_{l=1}^{2b_n} U_l^{(i)}  \right| O((n^{-1} \log n)^{1/2}) + O(n^{-1} \log n)\; .
\end{align*}
Since $U_l^{(i)}$ are Brownian motion increments, $U^{(i)}_l \overset{iid}{\sim} N(0,1)$ and by the classical strong law of large numbers, the above remains bounded with probability 1.  Similarly $(2b_n)^{-1} \sum_{l=1}^{2b_n} T_l^{(j)2}$ remains bounded with probability 1 as $n \to \infty$. Next, consider $R_n = \sum_{l=1}^{n} U_l^{(i)2}$. Since $U_l^{(i)} \sim N(0,1)$, $R_n \sim \chi^2_n$. Thus $R_{n}$ has a moment generating function and an application of Theorem \ref{thm:kmt} implies there exists a finite random variable $C_R$ such that, for sufficiently large $n$,
\begin{equation}
\label{eq:rn_kmt}
  \left|R_n - n - 2B^{(i)}(n)\right| < C_R \log n \; . 
\end{equation}
Consider 
\begin{align*}
 \left|R_n - R_{n-2b_n}  \right| &= \Big| \left(R_n - n - 2B^{(i)}(n) \right) - \left(R_{n - 2b_n} - (n - 2b_n) - 2B^{(i)}(n-2b_n) \right)\\
   & \quad -(n-2b_n) + n + 2B^{(i)}(n) - 2B^{(i)} (n - 2b_n)    \Big|\\
   & \leq \left| \left(R_n - n - 2B^{(i)}(n) \right) \right| +  \left|\left(R_{n - 2b_n} - (n - 2b_n) - 2B^{(i)}(n-2b_n) \right) \right| \\
& \quad + \left| 2b_n + 2B^{(i)}(n) - 2B^{(i)} (n - 2b_n) \right|\\
   & < C_R \log{n} + C_R \log{(n - b_n)} + 2b_n \\
& \quad + 2(1+\epsilon) \left( 2(2b_n) \left(\log{\dfrac{n}{2b_n}} + \log \log n \right)  \right)^{1/2} \quad\text{by \eqref{eq:rn_kmt} and Lemma \ref{bn.diff.bound}}\\
   & < 2C_R\log n + 2b_n + 4(1+\epsilon) (2b_n \log n)^{1/2}. \numberthis \label{eq:rn_diff}
\end{align*}

Finally, 
\begin{align*}
& \frac{1}{2b_n} \sum_{l = n - 2b_n +1}^{n} T_l^{(i)2}\\
  & = \dfrac{1}{2b_n} \ds \sum_{l = n-2b_n + 1}^{n} \left(U_l^{(i)} - \bar{B}^{(i)}_n  \right)^2\\
  & = \dfrac{1}{2b_n} \ds \sum_{l = n-2b_n +1}^{n} U_l^{(i)2} - \dfrac{2 \bar{B}^{(i)}_n}{2b_n} \ds \sum_{l = n-2b_n +1}^{n} U_l^{(i)}  + \left(\bar{B}^{(i)}_n  \right)^2\\
  & = \dfrac{1}{2b_n}(R_n - R_{n-2b_n}) - \dfrac{2}{n}B^{(i)}(n) \dfrac{1}{2b_n} \left(B^{(i)}(n) - B^{(i)}(n-2b_n)  \right) + \left(\dfrac{1}{n} B^{(i)}(n) \right)^2\\
  & < \dfrac{1}{2b_n} |R_n - R_{n-2b_n}| + \dfrac{2}{n}\left|B^{(i)}(n) \right|  \dfrac{1}{2b_n} \left|  B^{(i)}(n) - B^{(i)}(n-2b_n)  \right| + \left(\dfrac{1}{n} \left|B^{(i)}(n)  \right|\right)^2\\
  & < \dfrac{1}{2b_n} \left(2C_R\log n + 2b_n + 4(1+\epsilon) (2b_n \log n)^{1/2}  \right) \\
  & \quad  + \dfrac{2}{n} (1+ \epsilon) (2n \log \log n)^{1/2} \dfrac{1}{2b_n} (1+\epsilon) \left(2 (2b_n) \left( \log{\dfrac{n}{2b_n}} + \log \log n \right)  \right)^{1/2} \\
& \quad + \left(\dfrac{1}{n} (1+ \epsilon) (2n \log \log n)^{1/2} \right)^2 \quad \quad \text{by \eqref{eq:rn_diff} and Lemma \ref{bn.diff.bound}}\\
  & < C_Rb_n^{-1} \log n  + 1 + \dfrac{4(1 + \epsilon)(2b_n \log n)^{1/2}}{2b_n} + \dfrac{1}{nb_n}(1 + \epsilon)^2 (2n \log n)^{1/2} (8b_n \log n)^{1/2} \\
  & \quad + \dfrac{(1 + \epsilon)^2}{n} (2 \log n)\\
  & < C_R b_n^{-1} \log n + 1 + 2(1+\epsilon) (2b_n^{-1} \log n)^{1/2} + 4 (1+\epsilon)^2 \left(\dfrac{\log n}{n} \right)^{1/2} \left( b_n^{-1} \log n \right)^{1/2} + 2(1+\epsilon)^2 \dfrac{\log n}{n}.
\end{align*}
Since $b_n^{-1} \log n = O(1)$ as $n \to \infty$, the above term remains bounded with probability 1 as $n \to \infty$. Similarly, $(2b_n)^{-1} \sum_{l = n-2b_n+1}^{n} T_l^{(j)2} $ remains bounded with probability 1 as $n \to \infty$. The second term in the product in \eqref{eq:bounded_t} converges to 0 by Lemma~\ref{lem:window} and hence $\tilde{d}_{n,ij} \to 0$ with probability 1 as $n \to \infty$.
\end{proof}

Recall that $h(X_t) = Y_t^2$ for $t = 1,2, 3, \dots $, where the square is element-wise.
\begin{lemma}
\label{lemma:bounded.tails}

Let a strong invariance principle for $h$ hold as in \eqref{eq:multi_sip_h}. If Condition \ref{cond:bn} holds,  $b_n^{-1} \psi_h(n) \to  0$ and $b_n^{-1} \log n = O(1)$  as $n \to \infty$, then 
\[\dfrac{1}{b_n} \ds \sum_{k=1}^{b_n} h(X_k) \text{ and } \dfrac{1}{b_n} \ds \sum_{k = n-b_n +1}^{n} h(X_k),
\]
stay bounded with probability 1 as $n \to \infty$. 
\end{lemma}

\begin{proof}
Equation \eqref{eq:multi_sip_h} implies that $b_n^{-1} \sum_{k=1}^{b_n}
h(X_k) \to \E_F h$ if $b_n^{-1}\psi_h(b_n) \to 0$ as $n \to
\infty$. Since by assumption $b_n ^{-1} \psi_h(n) \to 0$ as $n \to
\infty$ and $\psi_h$ is increasing, $ b_n^{-1} \sum_{k=1}^{b_n}
h(X_k)$ remains bounded w.p. 1 as $n \to \infty$. Next, for all
$\epsilon > 0$ and sufficiently large $n(\epsilon)$,
\begin{align*}
&\dfrac{1}{b_n} \left \| \ds \sum_{k = n -b_n+1}^{n} h(X_k)  \right\|  \\
 = & \dfrac{1}{b_n} \left \| \ds \sum_{k = 1}^{n} h(X_k)    - \ds \sum_{k = 1}^{n - b_n } h(X_k) \right\|\\
 = & \dfrac{1}{b_n} \Bigg\| \ds \sum_{k=1}^{n} h(X_k) - n\E_{F} h + (n-b_n)\E_{F} h + b_n\E_{F}h  - L_h B(n) +  L_h B(n-b_n)\\
   & + L_h(B(n) - B(n-b_n)) - \ds \sum_{k=1}^{n - b_n} h(X_k) \Bigg\| \\
 \leq & \dfrac{1}{b_n} \left \| \ds \sum_{k=1}^{n} h(X_k) - n\E_{F}h - L_hB(n)  \right\| + \dfrac{1}{b_n} \left \|\ds \sum_{k=1}^{n-b_n} h(X_k) - (n-b_n)\E_{F}h - L_h B(n-b_n)   \right\| \\
 & + \dfrac{1}{b_n} \left \| L_h (B(n) - B(n-b_n))  + b_n \E_{F}h \right\|\\
 < & \dfrac{1}{b_n} D_h \psi_h(n) + \dfrac{1}{b_n}D_h \psi_h(n-b_n)+ \dfrac{1}{b_n} \left \| L_h (B(n) - B(n-b_n))\right\|   + \left \| \E_{F}h \right\|\quad \quad \text{by \eqref{eq:multi_sip_h}}\\
\leq & \dfrac{1}{b_n} D_h \psi_h(n) + \dfrac{1}{b_n}D_h \psi_h(n-b_n)+ \dfrac{1}{b_n} \|L_h\|  \left(\ds \sum_{i=1}^{p} \left|B^{(i)}(n) - B^{(i)}(n-b_n) \right|^2 \right)^{1/2} + \|\E_Fh\|\\
\leq & \dfrac{1}{b_n} D_h \psi_h(n) + \dfrac{1}{b_n}D_h \psi_h(n-b_n)+ \dfrac{1}{b_n} \|L_h\|  \left(\ds \sum_{i=1}^{p} \sup_{0 \leq s \leq b_n} \left|B^{(i)}(n) - B^{(i)}(n-s) \right|^2 \right)^{1/2} + \|\E_Fh\| \\
< & \dfrac{2}{b_n} D_h\psi_h(n) + \dfrac{p^{1/2}}{b_n} \|L_h\|  (1+\epsilon) \left(2b_n \left(\log{\dfrac{n}{b_n}} + \log{\log n}  \right) \right) ^{1/2} + \left \| \E_{F}h \right\| \quad \quad \text{by Lemma \ref{bn.diff.bound}}\\
< & \left \|\E_{F}h \right\| + \dfrac{2}{b_n} D_h \psi_h(n) + O((b_n^{-1} \log n)^{1/2} ).
\end{align*} 
Thus by the assumptions $b_n^{-1} \left \|\sum_{k= n-b_n+1}^{n} h(X_k)  \right\|$ stays bounded w.p. 1 as $n \to \infty$. 
\end{proof}

\begin{lemma}
\label{lemma:dn.0}
Suppose the strong invariance principles \eqref{eq:multi_sip} and \eqref{eq:multi_sip_h} hold. In addition,  suppose Conditions \ref{cond:lag.window} and \ref{cond:bn} hold and $n > 2b_n$, $b_n n^{-1} \sum_{k=1}^{b_n} k|\Delta_1 w_n(k)| \to 0$, $b_n^{-1} \psi(n) \to 0$, $b_n^{-1}\psi_h(n) \to 0$ as $n\to \infty$ and $b_n^{-1} \log n = O(1)$. Then,  $d_n \to 0$ with probability 1 as $n \to \infty$.
\end{lemma}

\begin{proof}
For $i,j = 1, \dots,p$, let $d_{n,ij}$ denote the $(i,j)$th element of the matrix $d_n$. We can follow the same steps as in Lemma \ref{lemma:dtilde.0} to obtain
\begin{align*}
|d_{n,ij}| & \leq \dfrac{1}{b_n} \left(  \dfrac{1}{2}\ds \sum_{l=1}^{2b_n} Z^{(i)2}_{l}  + \dfrac{1}{2}\ds \sum_{l=1}^{2b_n} Z^{(j)2}_{l} + \dfrac{1}{2}\ds \sum_{l=n-2b_n+1}^{n} Z^{(i)2}_{l}  + \dfrac{1}{2}\ds \sum_{l=n-2b_n+1}^{n} Z^{(j)2}_{l} \right) \times\\
 & \dfrac{b_n}{n} \left(  \sum_{t=1}^{b_n} |\Delta_1 w_n(t)| + 2\ds \sum_{s=1}^{b_n-1} \ds \sum_{t=1}^{b_n-s} |\Delta_1 w_n(s+t)| \right).
\end{align*}
The second term in the product converges to 0 by Lemma~\ref{lem:window}. It remains to show that the following remains bounded with probability 1 as $n \to \infty$,
\[\dfrac{1}{b_n} \left(  \dfrac{1}{2}\ds \sum_{l=1}^{2b_n} Z^{(i)2}_{l}  + \dfrac{1}{2}\ds \sum_{l=1}^{2b_n} Z^{(j)2}_{l} + \dfrac{1}{2}\ds \sum_{l=n-2b_n+1}^{n} Z^{(i)2}_{l}  + \dfrac{1}{2}\ds \sum_{l=n-2b_n+1}^{n} Z^{(j)2}_{l} \right). \] 
We have,
\begin{align*}
\dfrac{1}{2b_n} \ds \sum_{l=1}^{2b_n} Z_{l}^{(i)2} & = \dfrac{1}{2b_n}\ds \sum_{l=1}^{2b_n} \left( Y_l^{(i)} - \bar{Y}^{(i)}_n\right)^2 = \dfrac{1}{2b_n} \ds \sum_{l=1}^{2b_n} Y_l^{(i)2} - 2 \bar{Y}^{(i)}_{2b_{n}}\bar{Y}^{(i)}_n +  \left(\bar{Y}^{(i)}_n \right)^2.
\end{align*}
By the strong invariance principle for $g$, $\bar{Y}_{n}^{(i)} \to 0$, $\bar{Y}_{2b_n}^{(i)} \to 0$, and $  (\bar{Y}^{(i)}_n)^2 \to 0$ w.p. 1 as $n \to \infty$. By Lemma \ref{lemma:bounded.tails}, $(2b_n)^{-1} \sum_{l=1}^{2b_n} Y_l^{(i)2} $  remains bounded w.p. 1 as $n \to \infty$. Thus  $(2b_n)^{-1} \sum_{l=1}^{2b_n} Z_{l}^{(i)2}$ remains bounded w.p. 1 as $n \to \infty$.  Similarly $ (2b_n)^{-1} \sum_{l=1}^{2b_n} Z_{l}^{(j)2}$ stay bounded w.p. 1 as $n \to \infty$. Now consider
\begin{align*}
\dfrac{1}{2b_n} \ds \sum_{l=n-2b_n+1}^{n} Z_{l}^{(i)2} & = \dfrac{1}{2b_n}\ds \sum_{l=n-2b_n +1}^{n} \left( Y_l^{(i)} - \bar{Y}^{(i)}_n\right)^2\\
& = \dfrac{1}{2b_n} \ds \sum_{l=n-2b_n +1}^{n} Y_l^{(i)2} - 2 \bar{Y}^{(i)}_n \dfrac{1}{2b_n}\ds \sum_{l=n-2b_n +1}^{n} Y_l^{(i)} +  \left(\bar{Y}^{(i)}_n \right)^2. \numberthis \label{eq:z_lsquare}
\end{align*}
We will first show that $(2b_n)^{-1}\sum_{l=n-2b_n+1}^{n} Y_{l}^{(i)}$ remains bounded with probability 1. Let $\Sigma_{ii}$ denote the $i$th diagonal entry of $\Sigma$, then 
\begin{align*}
 &\frac{1}{2b_n}\sum_{l = n-2b_n+1}^{n} Y^{(i)}_l \\
   & = \dfrac{1}{2b_n} \left(\ds \sum_{l=1}^{n} Y_l^{(i)} - \ds \sum_{l=1}^{n-2b_n} Y_l^{(i)} \right)\\
  & = \dfrac{1}{2b_n}\left(\ds \sum_{l=1}^{n} Y_l^{(i)} - \sqrt{\Sigma_{ii}}B^{(i)}(n)  \right) - \dfrac{1}{2b_n}\left(\ds \sum_{l=1}^{n-2b_n} Y_l^{(i)} - \sqrt{\Sigma_{ii}} B^{(i)}(n - 2b_n)  \right)\\
& \quad  + \dfrac{1}{2b_n} \sqrt{\Sigma_{ii}} \left(B^{(i)}(n) - B^{(i)}(n - 2b_n) \right)\\
  & < \dfrac{1}{2b_n} \left|\sum_{l=1}^{n} Y_l^{(i)} - \sqrt{\Sigma_{ii}} B^{(i)}(n)   \right| + \dfrac{1}{2b_n}\left|  \sum_{l=1}^{n-2b_n} Y_l^{(i)} - \sqrt{\Sigma_{ii}} B^{(i)}(n - 2b_n)  \right| \\
  & \quad + \dfrac{1}{2b_n} \left|\sqrt{\Sigma_{ii}} \left(B^{(i)}(n) - B^{(i)}(n - 2b_n) \right)  \right|\\
  & < \dfrac{1}{2b_n}D \psi(n) + \dfrac{1}{2b_n}D \psi(n - 2b_n) + \dfrac{1}{2b_n} \sqrt{\Sigma_{ii}} \sup_{0 \leq s \leq 2b_n} | B^{(i)}(n) - B^{(i)}(n - s)| \quad \text{by \eqref{eq:multi_sip}}\\
  & < \dfrac{2D}{2b_n} \psi(n) + \sqrt{\Sigma_{ii}} \dfrac{1}{2b_n} (1+\epsilon) \left[ 2(2b_n) \left(\log \dfrac{n}{2b_n} + \log \log n \right) \right]^{1/2}  \quad \text{by Lemma \eqref{bn.diff.bound}}\\
  & < Db_n^{-1} \psi(n) + \sqrt{\Sigma_{ii}}(1+\epsilon) (2b_n^{-1} \log n)^{1/2}\\
  & = O(b_n^{-1} \psi(n)) + O((b_n^{-1} \log n)^{1/2}) \; .
\end{align*}
By the strong invariance principle for $g$, $\bar{Y}^{(i)}_n \to 0 $ and $  (\bar{Y}^{(i)}_n)^2 \to 0$ w.p. 1 as $n \to \infty$. By Lemma \ref{lemma:bounded.tails}, $(2b_n)^{-1} \sum_{l=n-2b_n+1}^{n} Y_l^{(i)2} $ remains bounded w.p. 1 as $n \to \infty$. Combining these results in \eqref{eq:z_lsquare}, $(2b_n)^{-1} \sum_{l=n-2b_n+1}^{n} Z_{l}^{(i)2}$ remains bounded w.p. 1 as $n \to \infty$. Similarly $(2b_n)^{-1} \sum_{l=n-2b_n+1}^{n} Z_{l}^{(j)2}$ remains bounded w.p. 1 as $n \to \infty$.
\end{proof}

\begin{lemma}
\label{lemma:avoidborel}
\citep{bill:2008} For a family of random variables $\{X_n: n\geq 1\}$, if $\E(|X_n|) \leq s_n$ where $s_n$ is a sequence such that  $\sum_{n=1}^{\infty} s_n < \infty$, then $X_n \to 0$ w.p. 1 as $n\to \infty$. 
\end{lemma}

\begin{lemma}
\label{whittle}
\citep{whit:1960} Let $R_1, \dots, R_n$ be i.i.d standard normal variables and \\ $A = \sum_{l=1}^{n} \sum_{k=1}^{n} a_{lk}R_lR_k$ where $a_{lk}$ are real coefficients, then for $c\geq 1$  and for some constant $K_c$, we have
\[\E[|A - \E A|^{2c}] \leq K_c \left( \sum_l \sum_k a_{lk}^2\right)^c.\]
\end{lemma}

\begin{lemma}
\label{lemma:sigma.tilde_wn_I}
Let Conditions \ref{cond:lag.window} and \ref{cond:bn} hold and assume that
\begin{enumerate}[(a)]
\item there exists a constant $c \geq 1$ such that $\sum_n (b_n/n)^c < \infty$,

\item $b_n n^{-1} \log n \to 0$ as $n \to \infty$,
\end{enumerate}

then $\widetilde{\Sigma}_{S} \to I_p$ w.p. 1 as $n \to \infty$.
\end{lemma}

\begin{proof}
Under the same conditions, Theorem 4.1 in \cite{dame:1991} shows $\widetilde{\Sigma}_{S,ii} \to 1$ as $n \to \infty$ w.p. 1. It is left to show that for all $i,j = 1,
\dots, p$, and $i \ne j$, $\widetilde{\Sigma}_{S,ij} \to 0 \, \text{ w.p. 1 as } n\to \infty$. Recall that
\begin{align*}
& \widetilde{\Sigma}_{S,ij}\\
& =  \ds \sum_{s = -(b_n-1)}^{b_n-1} w_n(s) \tilde{\gamma}_{n,ij}(s)\\
& =  \tilde{\gamma}_{n,ij}(0) + \dfrac{1}{n} \left[\ds \sum_{s = 1}^{b_n-1} w_n(s) \ds \sum_{t=1}^{n-s} (U^{(i)}_t - \bar{B}^{(i)}_n)(U^{(j)}_{t+s} - \bar{B}^{(j)}_n)\right. \\
& \quad + \left.  \sum_{s = -(b_n -1)}^{-1} w_n(s) \ds \sum_{t=1-s}^{n} (U^{(i)}_t - \bar{B}^{(i)}_n)(U^{(j)}_{t+s} - \bar{B}^{(j)}_n) \right]\\
& = \tilde{\gamma}_{n,ij}(0) + \dfrac{1}{n} \left[\ds \sum_{s = 1}^{b_n-1} w_n(s) \ds \sum_{t=1}^{n-s} (U^{(i)}_t - \bar{B}^{(i)}_n)(U^{(j)}_{t+s} - \bar{B}^{(j)}_n) \right. \\
& \quad + \left. \sum_{s = 1}^{b_n-1} w_n(s) \ds \sum_{t=1+s}^{n} (U^{(i)}_t - \bar{B}^{(i)}_n)(U^{(j)}_{t-s} - \bar{B}^{(j)}_n) \right]\\ 
& = \tilde{\gamma}_{n,ij}(0) + \ds \sum_{s = 1}^{b_n-1} w_n(s) \dfrac{1}{n} \Bigg[ \ds \sum_{t=1}^{n-s} \left(U^{(i)}_t U^{(j)}_{t+s} - \bar{B}^{(i)}_n U^{(j)}_{t+s} - \bar{B}^{(j)}_n U^{(i)}_t + \bar{B}^{(i)}_n \bar{B}^{(j)}_n   \right) \\
& \quad +  \ds \sum_{t=1+s}^{n} \left(U^{(i)}_t U^{(j)}_{t-s} - \bar{B}^{(i)}_n U^{(j)}_{t-s} - \bar{B}^{(j)}_n U^{(i)}_t + \bar{B}^{(i)}_n \bar{B}^{(j)}_n   \right)\Bigg]\, .
\end{align*}
Since 
\begin{align*}
 \sum_{t=1}^{n-s} U_{t+s}^{(j)} = B^{(j)}(n) - B^{(j)}(s), & \quad \quad\quad\quad \sum_{t=1}^{n-s} U_t^{(i)} = B^{(i)}(n-s),\\
\sum_{t = 1+s}^{n} U_{t-s}^{(j)} = B^{(j)}(n-s)\quad &\quad \text{ and }\quad \sum_{t=1+s}^{n} U^{(i)}_t = B^{(i)}(n) - B^{(i)}(s),
\end{align*}
we get $\widetilde{\Sigma}_{S,ij} = $
\begin{align*}
 & \quad = \tilde{\gamma}_{n,ij}(0) + \ds \sum_{s = 1}^{b_n-1} w_n(s) \Bigg[ \dfrac{1}{n}\ds \sum_{t=1}^{n-s} U^{(i)}_t U^{(j)}_{t+s}  -\dfrac{1}{n} \bar{B}^{(i)}_n( B^{(j)}(n) - B^{(j)}(s)  - \dfrac{1}{n} \bar{B}^{(j)}_n B^{(i)}(n-s) \\
& \quad + \left(\dfrac{n-s}{n} \right) \bar{B}^{(i)}_n \bar{B}^{(j)}_n   + \dfrac{1}{n} \ds \sum_{t=1+s}^{n} U_t^{(i)} U_{t-s}^{(j)} - \dfrac{1}{n} \bar{B}^{(i)}_n B^{(j)}(n-s) - \dfrac{1}{n} \bar{B}_n^{(j)} (B^{(i)}(n) - B^{(i)}(s))\\
& \quad + \left( \dfrac{n-s}{n} \right) \bar{B}^{(i)}_n \bar{B}^{(j)}_n \Bigg]\\
 & =  \tilde{\gamma}_{n,ij}(0) + \ds \sum_{s = 1}^{b_n-1} w_n(s) \Bigg[ \dfrac{1}{n}\ds \sum_{t=1}^{n-s} U^{(i)}_t U^{(j)}_{t+s} + \dfrac{1}{n} \ds \sum_{t=1+s}^{n} U_t^{(i)} U_{t-s}^{(j)} + 2\left( 1 - \dfrac{s}{n}\right) \bar{B}^{(i)}_n \bar{B}^{(j)}_n - 2\bar{B}^{(i)}_n\bar{B}^{(j)}_n \\
& \quad + \dfrac{1}{n} \bar{B}^{(i)}_nB^{(j)}(s) - \dfrac{1}{n}\bar{B}^{(j)}_n B^{(i)}(n-s) - \dfrac{1}{n} \bar{B}^{(i)}_nB^{(j)}(n-s) + \dfrac{1}{n} \bar{B}^{(j)}_nB^{(i)}(s) \Bigg]  \\ 
& =  \tilde{\gamma}_{n,ij}(0) + \ds \sum_{s = 1}^{b_n-1} w_n(s) \Bigg[ \dfrac{1}{n}\ds \sum_{t=1}^{n-s} U^{(i)}_t U^{(j)}_{t+s} + \dfrac{1}{n} \ds \sum_{t=1+s}^{n} U_t^{(i)} U_{t-s}^{(j)} - 2\left( 1 + \dfrac{s}{n}\right) \bar{B}^{(i)}_n \bar{B}^{(j)}_n \\
& \quad +  \bar{B}^{(i)}_n\bar{B}^{(j)}_n  - \dfrac{1}{n}\bar{B}^{(j)}_n B^{(i)}(n-s) + \bar{B}^{(i)}_n\bar{B}^{(j)}_n - \dfrac{1}{n} \bar{B}^{(i)}_nB^{(j)}(n-s) + \dfrac{1}{n} \bar{B}^{(i)}_nB^{(j)}(s) + \dfrac{1}{n} \bar{B}^{(j)}_nB^{(i)}(s) \Bigg]  \\ 
& =  \tilde{\gamma}_{n,ij}(0) + \ds \sum_{s = 1}^{b_n-1} w_n(s) \Bigg[ \dfrac{1}{n}\ds \sum_{t=1}^{n-s} U^{(i)}_t U^{(j)}_{t+s} + \dfrac{1}{n} \ds \sum_{t=1+s}^{n} U_t^{(i)} U_{t-s}^{(j)} - 2\left( 1+ \dfrac{s}{n}\right) \bar{B}^{(i)}_n \bar{B}^{(j)}_n\\
& \quad + \dfrac{1}{n} \bar{B}^{(j)}_n (B^{(i)}(n) - B^{(i)}(n-s))  + \dfrac{1}{n}\bar{B}^{(i)}_n (B^{(j)}(n) - B^{(j)}(n-s)) + \dfrac{1}{n} \bar{B}^{(i)}_n B^{(j)}(s)  \\
& \quad + \dfrac{1}{n} B^{(i)}(s)\bar{B}^{(j)}_n \Bigg]. \numberthis \label{eq:sigmatilde_wn_ij}
\end{align*}
We will show that each of the terms goes to 0 with probability 1 as $n \to \infty$.
\begin{enumerate}[1.]
\item 
\begin{align*}
  \tilde{\gamma}_{n,ij}(0) & = \dfrac{1}{n} \ds \sum_{t = 1}^{n} T_t^{(i)} T_t^{(j)}\\
  & = \dfrac{1}{n} \ds \sum_{t = 1}^{n} \left(U^{(i)}_t - \bar{B}^{(i)}_n \right) \left(U^{(j)}_t - \bar{B}^{(j)}_n\right)\\
  & = \dfrac{1}{n} \ds \sum_{t=1}^{n} U_t^{(i)} U_t^{(j)} - \bar{B}^{(j)}_n \dfrac{1}{n} \ds \sum_{t=1}^{n} U_t^{(i)} - \bar{B}^{(i)}_n \dfrac{1}{n} \ds \sum_{t=1}^{n} U_t^{(j)} + \bar{B}^{(i)}_n \bar{B}^{(j)}_n.  \numberthis \label{eq:gamma.tilde.0}
\end{align*}
We will show that each of the terms in \eqref{eq:gamma.tilde.0} goes to 0 with probability 1, as $n \to \infty$. First, we will use Lemma \ref{whittle} to show that $n^{-1} \sum_{t=1}^{n} U_t^{(i)} U_t^{(j)} \to 0$ with probability 1 as $n \to \infty$. Define
\[
R_1 = U_1^{(i)}, \, R_2 = U_2^{(i)}, \dots, \, R_{n} = U_{n}^{(i)},  \, R_{n + 1} = U_1^{(j)}, \dots, \,  R_{2n} = U_{n}^{(j)}.  
\]
Thus, $\{R_i: 1 \leq i \leq 2n\}$ is an i.i.d sequence of normally distributed random variables. Define for $1 \leq l,k, \leq 2n$,
\[ a_{lk} = \begin{cases}
  \dfrac{1}{n}, &\text{ if }1  \leq l \leq n \text{ and }k = l + n\\
  0 &\text{ otherwise }\, .
\end{cases}  \]

Then,
\[
A  := \ds \sum_{l=1}^{2n} \ds \sum_{k=1}^{2n} a_{lk} R_l R_k
 = \ds \sum_{t = 1}^{n} \dfrac{1}{n} U_t^{(i)} U_t^{(j)} \, .
\]

By Lemma \ref{whittle}, for all $c \geq 1$ there exists $K_c$ such that
\begin{align*}
\E[|A - \E A|^{2c}] & \leq K_c \left( \sum_l \sum_k a_{lk}^2\right)^c \; .
\end{align*}
Since $i \ne j, \E [A] = 0$,
\[
\E \left(\left| \dfrac{1}{n} \ds \sum_{t=1}^{n} U^{(i)}_t U^{(j)}_t \right|^{2c}  \right)  \leq K_c \left( \sum_{l = 1}^{2n} \sum_{k = 1}^{2n} a_{lk}^2\right)^c
 = K_c \left(\ds \sum_{t=1}^{n} \dfrac{1}{n^2}  \right)^c 
 = K_c \,  n^{-c}\, .
\]

Note that $\sum_{n=0}^{\infty} n^{-c} < \infty$ for all $c > 1$, hence by Lemma \ref{lemma:avoidborel}, $n^{-1} \sum_{t=1}^{n} U_t^{(i)} U_t^{(j)} \to 0$ with probability 1 as $n \to \infty$. Next in \eqref{eq:gamma.tilde.0}, 
\begin{align*}
    \bar{B}_n^{(j)} \dfrac{1}{n} \ds \sum_{t=1}^{n} U_t^{(i)} & \leq  \dfrac{1}{n} \left|B^{(j)} (n) \right| \left|  \dfrac{1}{n} \sum_{t=1}^{n} U_t^{(i)} \right | \\ 
    & < \dfrac{1}{n}(1 + \epsilon) \sqrt{2 n \log \log n} \left |\dfrac{1}{n} \sum_{t=1}^{n} U_t^{(i)}  \right| \quad \text{by Lemma \ref{bn.diff.bound}}\\
    & < \sqrt{2} (1+\epsilon) \left(\dfrac{\log n}{n}  \right)^{1/2}  \left| \dfrac{1}{n} \sum_{t=1}^{n} U_t^{(i)} \right| \; .
\end{align*}
By the classical SLLN 
\[
\left| \frac{1}{n} \sum_{t=1}^{n} U_t^{(i)} \right| \to 0 \quad \text{ w.p. 1 as } n \to \infty.
\]
Similarly,
\[
\bar{B}^{(i)}_n  \frac{1}{n}\sum_{t=1}^{n}  U_t^{(j)} \to 0 \quad \text{ w.p. 1 as } n \to \infty \; .
\]
Finally, 
\begin{align*}
  \bar{B}^{(i)}_n \bar{B}^{(j)}_n & \leq \dfrac{1}{n^2}\left|
        B^{(i)}(n) \right| \left| B^{(j)} (n) \right| \\ & <
        \dfrac{1}{n^2} (1 + \epsilon)^2 (2 n \log \log n) \quad
        \text{by Lemma \ref{bn.diff.bound} }\\ 
        & < 2 (1 + \epsilon)^2
        \left(\dfrac{\log n}{n} \right)\\ & \to 0 \text{ as } n \to
        \infty \; .
\end{align*}

Thus, $\tilde{\gamma}_{n, ij}(0) \to 0$ with probability 1 as $n \to \infty$. \\

\item Now consider the term $\sum_{s = 1}^{b_n-1} w_n(s)  n^{-1} \sum_{t=1}^{n-s} U^{(i)}_t U^{(j)}_{t+s}$.  Define
\[
R_1 = U_1^{(i)}, \, R_2 = U_2^{(i)}, \dots, \, R_{n} = U_{n}^{(i)},  \, R_{(n+1)} = U_1^{(j)}, \dots, \,  R_{2n} = U_{n}^{(j)}  \, .
\]
Thus, $\{R_i: 1 \leq i \leq 2n\}$ is an i.i.d sequence of normally distributed random variables. Next, define for $1 \leq l,k \leq 2n$
\[ a_{lk} =   \begin{cases}
  \frac{1}{n}w_n(k-(n+l)), &\text{ if }1 \leq l \leq n-1, n+2 \leq k \leq 2n,  \text{ and } 1 \leq k-(n+l) \leq b_n-1\\
  0 & \text{otherwise}.
 \end{cases}\]

Then, 
\begin{align*}
A :=  \sum_{l=1}^{2n} \sum_{k=1}^{2n} a_{lk} R_lR_k\\
& =  \sum_{l=1}^{n-1} \sum_{k = n+2}^{2n} I\{1 \leq k - (n+l) \leq b_n-1\}\dfrac{1}{n} w_n(k-(n+l))R_lR_k\\
  & = \ds \sum_{l=1}^{n-1} \sum_{s = 2-l}^{n-l} I\{1 \leq s \leq b_n - 1 \} \dfrac{1}{n} w_n(s) R_lR_{n+l+s} \quad \quad\text{Letting $k - (n+l) = s$ }\\
  & = \ds \sum_{s = 1}^{n-1} \ds \sum_{l=1}^{n-s}I\{1 \leq s \leq b_n - 1 \} \dfrac{1}{n} w_n(s) R_lR_{n+l+s} \\
  & \quad + \ds \sum_{s=(3-n)}^{0} \ds \sum_{l = (2-s)}^{n-1} I\{1 \leq s \leq b_n - 1 \} \dfrac{1}{n} w_n(s) R_lR_{n+l+s}\\
  & = \ds \sum_{s = 1}^{b_n-1} \ds \sum_{l=1}^{n-s} \dfrac{1}{n} w_n(s) U_l^{(i)}U_{l+s}^{(j)} \quad \quad \text{since $n > 2b_n \geq 2$}\\
  & = \ds \sum_{s = 1}^{b_n-1} \ds \sum_{l=1}^{n-s} \dfrac{1}{n} w_n(s) U_l^{(i)}U_{l+s}^{(j)}.
\end{align*}

Using Lemma \ref{whittle}, for $c \geq 1$ and  some constant $K_{c}$, 
\[\E \left[ \left(\ds \sum_{s = 1}^{b_n-1} w_n(s) \dfrac{1}{n}\ds \sum_{t=1}^{n-s} U^{(i)}_t U^{(j)}_{t+s} \right)^{2c}  \right] \leq K_c \left(\sum_l \sum_k a_{lk}^2  \right)^{c},\] where 
\[ \ds \sum_{l} \ds \sum_{k} a_{lk}^2 =  \ds \sum_{s=1}^{b_n-1} \ds \sum_{t=1}^{n-s} \dfrac{1}{n^2}w_n^2(s) = \dfrac{1}{n^2} \ds \sum_{s=1}^{b_n-1}  (n-s)w_n^2(s) \leq \dfrac{n}{n^2} \ds \sum_{s=1}^{b_n-1} 1  \leq \dfrac{b_n}{n}. \]

Thus, by Assumption (a) and Lemma \ref{lemma:avoidborel},
\[\ds \sum_{s = 1}^{b_n-1} w_n(s) \dfrac{1}{n}\ds \sum_{t=1}^{n-s} U^{(i)}_t U^{(j)}_{t+s} \to 0 \text{ w.p. 1 as } n \to \infty.  \]

\item 
 By letting $t-s = l$,
\[
\ds \sum_{s = 1}^{b_n-1} w_n(s)  \dfrac{1}{n}\ds \sum_{t=1+s}^{n} U^{(i)}_t U^{(j)}_{t-s}  = \ds \sum_{s=1}^{b_n-1}w_n(s) \dfrac{1}{n} \ds \sum_{l=1}^{n-s} U_{l+s}^{(i)} U_l^{(j)}.
\]

This is similar to the previous part with just the $i$ and $j$ components interchanged. A similar argument will lead to $ \sum_{s = 1}^{b_n-1} w_n(s) {n}^{-1}  \sum_{t=1+s}^{n} U^{(i)}_t U^{(j)}_{t-s} \to 0$ with probability 1 as $n\to \infty$.

\item
\begin{align*}
& \ds \sum_{s=1}^{b_n-1} 2w_n(s) \left( 1+ \dfrac{s}{n}\right) \bar{B}^{(i)}_n \bar{B}^{(j)}_n \\
&\leq   \left| \ds \sum_{s=1}^{b_n-1} 2w_n(s) \left( 1+ \dfrac{s}{n}\right) \bar{B}^{(i)}_n \bar{B}^{(j)}_n\right| \\
& \leq \ds \sum_{s=1}^{b_n-1} 2|w_n(s)| \left( 1+ \dfrac{s}{n}\right) \left|\bar{B}^{(i)}_n\right|  \left| \bar{B}^{(j)}_n\right|\\
& \leq \dfrac{2}{n^2} \ds \sum_{s=1}^{b_n-1}  \left( 1+ \dfrac{s}{n}\right) \left|{B}^{(i)}(n)\right|  \left| {B}^{(j)}(n)\right| \quad \quad \text{since } |w_n(s)| \leq 1 \\
& < \dfrac{2}{n^2} (1+ \epsilon)^2 2n \log \log n \ds \sum_{s=1}^{b_n-1} \left( 1+ \dfrac{s}{n}\right) \quad \quad \text{by Lemma \ref{bn.diff.bound}}\\
& < 4(1+\epsilon)^2 n^{-1} \log n \ds \sum_{s=1}^{b_n-1} 2\\
& \leq 8(1+\epsilon)^2 b_n n^{-1} \log n\\
& \to 0 \, .
\end{align*} 

\item Next
\begin{align*}
& \sum_{s=1}^{b_n-1} w_n(s) \dfrac{1}{n} \bar{B}^{(j)}_n \left(B^{(i)}(n) - B^{(i)}(n-s) \right)\\
&\leq \left| \ds \sum_{s=1}^{b_n-1} w_n(s) \dfrac{1}{n} \bar{B}^{(j)}_n \left(B^{(i)}(n) - B^{(i)}(n-s) \right) \right|\\
 & \leq \ds \sum_{s=1}^{b_n-1} \dfrac{1}{n^2} \left|{B}^{(j)}(n) \right| \left| B^{(i)}(n) - B^{(i)}(n-s)  \right| \quad \quad \text{since } |w_n(s)| \le 1 \\
& \leq  \dfrac{1}{n^2} \left|B^{(j)}(n) \right| \ds \sum_{s=1}^{b_n-1} \sup_{0 \leq m \leq b_n} |B^{(i)}(n) - B^{(i)}(n-m)| \\
& < \dfrac{1}{n^2} \left( (1 + \epsilon) (2n \log \log n)^{1/2} \right) (1+\epsilon) \left(2  b_n \left(\log \dfrac{n}{b_n} + \log \log n  \right)  \right)^{1/2}   \ds \sum_{s=1}^{b_n-1} 1 \quad  \text{by Lemma \ref{bn.diff.bound} }\\
& < 2^{1/2}(1 + \epsilon)^2 \dfrac{1}{n^2} (n \log n)^{1/2} (4b_n \log n)^{1/2} b_n\\
& < 2^{3/2} (1 +\epsilon)^2 \left(\dfrac{b_n}{n}  \right)^{1/2} n^{-1} b_n \log n\\
& \to 0 \; .
\end{align*}

\item  Similar to the previous term, but exchanging the $i$ and $j$ indices, 
\[\ds \sum_{s=1}^{b_n-1} w_n(s) \dfrac{1}{n} \bar{B}^{(i)}_n \left(B^{(j)}(n) - B^{(j)}(n-s) \right) \to 0 \text{ with probability 1 as } n\to \infty. \]

\item 
\begin{align*}
& \ds \sum_{s=1}^{b_n-1} w_n(s) \dfrac{1}{n} \bar{B}^{(i)}_n B^{(j)}(s)\\
& \leq  \left| \ds \sum_{s=1}^{b_n-1} w_n(s) \dfrac{1}{n} \bar{B}^{(i)}_n B^{(j)}(s) \right|\\
 & \leq \ds \sum_{s=1}^{b_n-1} |w_n(s)| \dfrac{1}{n} \left|\bar{B}^{(i)}_n \right|  \left| B^{(j)}(s) \right|\\
& \leq \dfrac{1}{n^2} \left|{B}^{(i)}(n) \right| \ds \sum_{s=1}^{b_n-1}  \left| B^{(j)}(s) \right|\quad \quad \text{since }  |w_n(s)| \le 1\\
& < \dfrac{1}{n^2} (1 + \epsilon) (2n \log \log n)^{1/2}  \ds \sum_{s=1}^{b_n-1} \sup_{1 \leq m \leq b_n} |B^{(j)}(m)|  \quad \quad \text{by Lemma \ref{bn.diff.bound}} \\
& < \dfrac{1}{n^2} (1 + \epsilon) (2n \log \log n)^{1/2}  \sup_{1 \leq m \leq b_n} |B^{(j)}(m + 0) - B^{(j)}(0)|   \ds \sum_{s=1}^{b_n-1} 1 \\
& <  \dfrac{b_n}{n^2} (1 + \epsilon) (2n \log \log n)^{1/2}  \sup_{0 \leq t \leq n-b_n} \sup_{0 \leq m \leq b_n} |B^{(j)}(t+m) - B^{(j)}(t)|\\
& < \dfrac{b_n}{n^2} (1 + \epsilon) (2n \log \log n)^{1/2} (1+\epsilon) \left(2b_n \left(\log \dfrac{n}{b_n} + \log \log n  \right) \right)^{1/2}\\
& = 2^{3/2} ( 1+ \epsilon)^2 \dfrac{b_n^{1/2}}{n^{1/2}} b_n n^{-1} \log n\\
& \to 0 \, .
\end{align*}

\item Similar to the previous term, by exchanging the $i$ and $j$ index,
\[ \ds \sum_{s=1}^{b_n-1} w_n(s) \dfrac{1}{n} \bar{B}^{(j)}_n B^{(i)}(s)  \to 0 \text{ w.p. 1 as } n \to \infty.  \]
\end{enumerate}

Since each term in \eqref{eq:sigmatilde_wn_ij} goes to 0, we get that
\[ \widetilde{\Sigma}_{S,ij} \to 0 \text{ w.p. 1 as } n \to \infty. \] 
\end{proof}

\begin{lemma}
\label{lemma:sig.tilde_to_I}
Let Conditions  \ref{cond:lag.window} and \ref{cond:bn} hold. In addition, suppose there exists a constant $c \geq 1$ such that $\sum_{n} (b_n/n)^c < \infty$, $ n > 2b_n$, $b_n n^{-1} \log n \to 0$, and $b_n n^{-1}  \sum_{k=1}^{b_n} k |\Delta_1 w_n(k) | \to 0$, then $\widetilde{\Sigma}_{w,n} \to I_p$ w.p. 1 as $n\to \infty$ where $I_p$ is the $p \times p$ identity matrix. 
\end{lemma}

\begin{proof}
The result follows from Lemmas \ref{lemma:dn.tilde}, \ref{lemma:dtilde.0}  and  \ref{lemma:sigma.tilde_wn_I}.
\end{proof}

The following corollary is an immediate consequence of the previous lemma.
\begin{corollary}
\label{lsigmal}
Under the conditions of Lemma \ref{lemma:sig.tilde_to_I}, $L \widetilde{\Sigma}_{w,n} L^T \to LL^T = \Sigma$ w.p. 1 as $n \to \infty$.
\end{corollary}

\begin{lemma}
\label{lemma:wn.to.sigma}
Suppose \eqref{eq:multi_sip} holds and Conditions \ref{cond:lag.window} and \ref{cond:bn} hold. If as $n \to \infty$,
\begin{eqnarray}
b_n \psi(n)^2 \log n \left( \ds \sum_{k=1}^{b_n} |\Delta_2 w_n(k)| \right)^2 \to 0 & \text{ and} \label{assbn1}\\
\psi(n)^2 \ds \sum_{k=1}^{b_n} |\Delta_2 w_n(k)| \to 0, \label{assbn2}&
\end{eqnarray}
then $\widehat{\Sigma}_{w,n} \to \Sigma$ w.p. 1.
\end{lemma}

\begin{proof}
For $i ,j = 1, \dots, p$, let $\Sigma_{ij}$ and $\widehat{\Sigma}_{w,ij}$ denote the $(i,j)$th element of $\Sigma$ and $\widehat{\Sigma}_{w,n}$ respectively. Recall
\[\widehat{\Sigma}_{w,ij} = \dfrac{1}{n} \ds \sum_{l=0}^{n-b_n} \sum_{k=1}^{b_n} k^2 \Delta_2 w_n(k) \left[\bar{Y}^{(i)}_l(k) - \bar{Y}^{(i)}_n\right] \left[\bar{Y}^{(j)}_l(k) - \bar{Y}^{(j)}_n \right].\] 
We have 
\begin{align*}
& |\widehat{\Sigma}_{w,ij} - \Sigma_{ij}|\\
& =  \left|\dfrac{1}{n} \ds \sum_{l=0}^{n-b_n} \sum_{k=1}^{b_n} k^2 \Delta_2 w_n(k) \left[\bar{Y}^{(i)}_l(k) - \bar{Y}^{(i)}_n\right] \left[\bar{Y}^{(j)}_l(k) - \bar{Y}^{(j)}_n \right] - \Sigma_{ij} \right|\\
& =  \bigg|\dfrac{1}{n} \ds \sum_{l=0}^{n-b_n} \sum_{k=1}^{b_n} k^2 \Delta_2 w_n(k) \left[\bar{Y}^{(i)}_l(k) - \bar{Y}^{(i)}_n \pm \bar{C}^{(i)}_l(k) \pm \bar{C}^{(i)}_n\right] \left[\bar{Y}^{(j)}_l(k) - \bar{Y}^{(j)}_n \pm \bar{C}^{(j)}_l(k) \pm \bar{C}^{(j)}_n\right]\\
 & \quad  - \Sigma_{ij} \bigg|\\
& = \bigg|\dfrac{1}{n} \ds \sum_{l=0}^{n-b_n} \sum_{k=1}^{b_n} k^2 \Delta_2 w_n(k) \left[\left( \bar{Y}^{(i)}_l(k) - \bar{C}^{(i)}_l(k) \right) + \left(\bar{C}^{(i)}_l(k) - \bar{C}_{n}^{(i)}   \right) - \left( \bar{Y}^{(i)}_n - \bar{C}^{(i)}_n  \right)  \right] \\ & \left[\left( \bar{Y}^{(j)}_l(k) - \bar{C}^{(j)}_l(k) \right) + \left(\bar{C}^{(j)}_l(k) - \bar{C}_{n}^{(j)}   \right) - \left( \bar{Y}^{(j)}_n - \bar{C}^{(j)}_n  \right)  \right] - \Sigma_{ij}  \bigg|\\
 \leq & \left| \dfrac{1}{n} \ds \sum_{l=0}^{n-b_n} \ds \sum_{k=1}^{b_n} k^2 \Delta_2 w_n(k) \left(\bar{C}^{(i)}_l(k) - \bar{C}^{(i)}_n \right) \left(\bar{C}^{(j)}_l(k) - \bar{C}^{(j)}_n \right) - \Sigma_{ij}  \right|\\
 &\quad  +  \dfrac{1}{n} \ds \sum_{l=0}^{n-b_n} \ds \sum_{k=1}^{b_n} k^2 |\Delta_2 w_n(k)| \bigg[ \left|\left(\bar{Y}^{(i)}_l(k)   - \bar{C}^{(i)}_l(k) \right)\left(\bar{Y}^{(j)}_l(k)   - \bar{C}^{(j)}_l(k) \right)  \right| \\
 &\quad  + \left|\left(\bar{Y}^{(i)}_l(k)   - \bar{C}^{(i)}_l(k) \right) \left(\bar{Y}^{(j)}_n - \bar{C}^{(j)}_n \right)   \right|  \\
 & \quad +  \left| \left(\bar{Y}^{(i)}_l(k)   - \bar{C}^{(i)}_l(k) \right) \left(  \bar{C}^{(j)}_l(k) - \bar{C}^{(j)}_n \right) \right|  + \left|\left(\bar{Y}^{(i)}_n - \bar{C}^{(i)}_n \right) \left(\bar{Y}^{(j)}_n - \bar{C}^{(j)}_n \right)   \right| \\
 & \quad+  \left|\left(\bar{Y}^{(i)}_n - \bar{C}^{(i)}_n \right)\left(\bar{Y}^{(j)}_l(k) - \bar{C}^{(j)}_l(k) \right)  \right| + \left|\left(\bar{Y}^{(i)}_n - \bar{C}^{(i)}_n \right) \left(\bar{C}^{(j)}_l(k) - \bar{C}^{(j)}_n \right) \right|\\
 & \quad+ \left|\left(\bar{C}^{(i)}_l(k) - \bar{C}^{(i)}_n \right)\left(\bar{Y}^{(j)}_l(k) - \bar{C}^{(j)}_l(k)  \right)   \right| + \left|\left(\bar{C}^{(i)}_l(k) - \bar{C}^{(i)}_n \right)\left(\bar{Y}^{(j)}_n - \bar{C}^{(j)}_n  \right)  \right|\bigg] \numberthis \label{eq:sigma_diff}.
\end{align*}

We will show that each of the nine terms in \eqref{eq:sigma_diff} goes to 0 with probability 1 as $n \to \infty$. To do that, let us first establish a useful inequality. From \eqref{eq:multi_sip}, for any component $i$, and sufficiently large $n$, 
\begin{equation}
\label{one.eq}
\left|\ds \sum_{t = 1}^{n}Y_t^{(i)} - C^{(i)}(n)  \right| < D\psi(n).
\end{equation}

\begin{enumerate}[1.]
\item $\left|\dfrac{1}{n} \ds \sum_{l=0}^{n-b_n} \ds \sum_{k=1}^{b_n} k^2 \Delta_2 w_n(k)   \left(\bar{C}^{(i)}_l(k) - \bar{C}^{(i)}_n \right) \left(\bar{C}^{(j)}_l(k) - \bar{C}^{(j)}_n \right) - \Sigma_{ij}  \right|$

Notice that 
\[\dfrac{1}{n} \ds \sum_{l=0}^{n-b_n} \ds \sum_{k=1}^{b_n} k^2 \Delta_2 w_n(k) \left(\bar{C}^{(i)}_l(k) - \bar{C}^{(i)}_n \right) \left(\bar{C}^{(j)}_l(k) - \bar{C}^{(j)}_n \right), \]
is equivalent to the $ij$th entry in $L \widetilde{\Sigma}_{w,n}L^T$. Then by Corollary \ref{lsigmal}
\[
\left|\dfrac{1}{n} \ds \sum_{l=0}^{n-b_n} \ds \sum_{k=1}^{b_n} k^2 \Delta_2 w_n(k)   \left(\bar{C}^{(i)}_l(k) - \bar{C}^{(i)}_n \right) \left(\bar{C}^{(j)}_l(k) - \bar{C}^{(j)}_n \right) - \Sigma_{ij}  \right| \to 0 \quad  \text{ as } n \to \infty ~~ \text{w.p. 1}\,.
\]

\item $\dfrac{1}{n} \ds \sum_{l=0}^{n-b_n} \ds \sum_{k=1}^{b_n} k^2 |\Delta_2 w_n(k)| \left|\left(\bar{Y}^{(i)}_l(k)   - \bar{C}^{(i)}_l(k) \right)\left(\bar{Y}^{(j)}_l(k)   - \bar{C}^{(j)}_l(k) \right)  \right|$

Note that for any component $i$,
\begin{align*}
\left| k\left(\bar{Y}^{(i)}_l(k)   - \bar{C}^{(i)}_l(k) \right) \right| & = \left|\ds \sum_{t=1}^{k} Y^{(i)}_{l+t} - C^{(i)}(k+l) + C^{(i)}(l)  \right|\\
& = \left| \ds \sum_{t=1}^{k+l} Y^{(i)}_{t} - \ds \sum_{t=1}^{l} Y_t^{(i)} - C^{(i)}(k+l) + C^{(i)}(l)  \right|\\
& < \left| \ds\sum_{t=1}^{l+k}Y_t^{(i)} - C^{(i)}(k+l)  \right| + \left|\ds \sum_{t=1}^{l} Y_t^{(i)} - C^{(i)}(l)  \right|\\
& \leq D \psi(l+k) + D \psi(l) \quad \quad \text{by \eqref{one.eq}}\\
& \leq 2D \psi(n)  \quad \quad \text{since $l+k \leq n$}. \numberthis \label{yk}
\end{align*}

By \eqref{yk},
\begin{align*}
& \dfrac{1}{n} \ds \sum_{l=0}^{n-b_n} \ds \sum_{k=1}^{b_n} k^2 |\Delta_2 w_n(k)| \left|\left(\bar{Y}^{(i)}_l(k)   - \bar{C}^{(i)}_l(k) \right)\left(\bar{Y}^{(j)}_l(k)   - \bar{C}^{(j)}_l(k) \right)  \right|\\ 
& < \dfrac{1}{n} \ds \sum_{l=0}^{n-b_n} \ds \sum_{k=1}^{b_n} |\Delta_2 w_n(k)|\left(2D \psi(n)  \right)^2\\
& = 4D^2 \left(\dfrac{n - b_n + 1}{n}\right) \psi(n)^2 \ds \sum_{k=1}^{b_n} |\Delta_2 w_n(k)|\\
& \to 0 \quad  \text{ as } n \to \infty ~~ \text{with probability 1}\,.
\end{align*}

\item $\dfrac{1}{n} \ds \sum_{l=0}^{n-b_n} \ds \sum_{k=1}^{b_n} k^2 |\Delta_2 w_n(k)| \left|\left(\bar{Y}^{(i)}_l(k)   - \bar{C}^{(i)}_l(k) \right) \left(\bar{Y}^{(j)}_n - \bar{C}^{(j)}_n \right)   \right| $

Note that for any component $i$, using \eqref{one.eq},
\begin{equation}
\label{yn}
\left| \bar{Y}^{(i)}_n - \bar{C}^{(i)}_n \right|  = \dfrac{1}{n}
\left|\ds \sum_{t=1}^{n}Y_t^{(i)} - C^{(i)}(n) \right| <
\dfrac{1}{n} D \psi(n)\, .
\end{equation}

By \eqref{yk} and \eqref{yn},
\begin{align*}
& \dfrac{1}{n} \ds \sum_{l=0}^{n-b_n} \ds \sum_{k=1}^{b_n} k^2 |\Delta_2 w_n(k)| \left|\left(\bar{Y}^{(i)}_l(k)   - \bar{C}^{(i)}_l(k) \right) \left(\bar{Y}^{(j)}_n - \bar{C}^{(j)}_n \right)   \right|\\
& < \dfrac{1}{n} \ds \sum_{l=0}^{n-b_n} \ds \sum_{k=1}^{b_n} k |\Delta_2 w_n(k)| \left(2D \psi(n) \right) \left(\dfrac{1}{n} D \psi(n) \right)\\
& = 2D^2 \psi(n)^2 \left(\dfrac{n - b_n + 1}{n}  \right) \dfrac{1}{n} \ds \sum_{k=1}^{b_n} k |\Delta_2 w_n(k)|\\
& \leq 2D^2 \psi(n)^2 \left(\dfrac{n - b_n + 1}{n}  \right) \dfrac{b_n}{n} \ds \sum_{k=1}^{b_n} |\Delta_2 w_n(k)|\\
& \to 0  \quad  \text{ as } n \to \infty ~~ \text{with probability 1}\,.
\end{align*}

\item Now
\begin{align*}
& \dfrac{1}{n} \ds \sum_{l=0}^{n-b_n} \ds \sum_{k=1}^{b_n} k^2 |\Delta_2 w_n(k)| \left|\left(\bar{Y}^{(i)}_l(k)   - \bar{C}^{(i)}_l(k) \right) \left(\bar{C}^{(j)}_l(k) - \bar{C}^{(j)}_n \right) \right| \\
 \leq & \, \dfrac{1}{n} \ds \sum_{l=0}^{n-b_n} \ds \sum_{k=1}^{b_n} k^2 |\Delta_2 w_n(k)| \left|\left(\bar{Y}^{(i)}_l(k)   - \bar{C}^{(i)}_l(k) \right) \bar{C}^{(j)}_l(k) \right|\\
 & +   \dfrac{1}{n} \ds \sum_{l=0}^{n-b_n} \ds \sum_{k=1}^{b_n} k^2 |\Delta_2 w_n(k)| \left|\left(\bar{Y}^{(i)}_l(k)   - \bar{C}^{(i)}_l(k) \right) \bar{C}^{(j)}_n \right|.
\end{align*}

We will show that both parts of the sum converge to 0 with probability 1 as $n \to \infty$. Consider the first sum.  
\begin{align*}
& \dfrac{1}{n} \ds \sum_{l=0}^{n-b_n} \ds \sum_{k=1}^{b_n} k^2 |\Delta_2 w_n(k)| \left|\left(\bar{Y}^{(i)}_l(k)   - \bar{C}^{(i)}_l(k) \right) \bar{C}^{(j)}_l(k) \right|\\
& \le \dfrac{1}{n} \ds \sum_{l=0}^{n-b_n} \ds \sum_{k=1}^{b_n} k^2
  |\Delta_2 w_n(k)| \left|\bar{Y}^{(i)}_l(k) -
  \bar{C}^{(i)}_l(k) \right| \left| \bar{C}^{(j)}_l(k)
  \right|\\ & \leq \dfrac{1}{n} \ds \sum_{l=0}^{n-b_n} \ds
  \sum_{k=1}^{b_n} k |\Delta_2 w_n(k)| \left( 2D \psi(n) \right)
  \left( 2 (1+\epsilon) \sqrt{b_n \Sigma_{ii}} \dfrac{1}{k} (\log
  n)^{1/2} \right)\quad \text{by \eqref{diff.bound} and
    \eqref{yk}}\\ & = \left(\dfrac{n-b_n+1}{n} \right) 4 D (1+
  \epsilon) \sqrt{\Sigma_{ii}b_n \log n}\, \psi(n) \ds
  \sum_{k=1}^{b_n} |\Delta_2 w_n(k)|\\ & \to 0 \text{ by Condition
    \ref{cond:bn} and \eqref{assbn1}}\; .
\end{align*}

The second part is
\begin{align*}
&\dfrac{1}{n} \ds \sum_{l=0}^{n-b_n} \ds \sum_{k=1}^{b_n} k^2 |\Delta_2 w_n(k)| \left|\left(\bar{Y}^{(i)}_l(k)   - \bar{C}^{(i)}_l(k) \right) \bar{C}^{(j)}_n \right|\\
& = \dfrac{1}{n} \ds \sum_{l=0}^{n-b_n} \ds \sum_{k=1}^{b_n} k^2 |\Delta_2 w_n(k)| \left|\bar{Y}^{(i)}_l(k)   - \bar{C}^{(i)}_l(k) \right|  \left| \bar{C}^{(j)}_n \right|\\
& \leq \dfrac{1}{n} \ds \sum_{l=0}^{n-b_n} \ds \sum_{k=1}^{b_n} k |\Delta_2 w_n(k)| \left( 2D \psi(n) \right) \left( \dfrac{1}{n}(1+ \epsilon) \left[2n \Sigma_{ii} \log \log n  \right]^{1/2} \right) \quad \quad \text{by  \eqref{yk} and \eqref{n.bound}}\\
& < \left(\dfrac{n - b_n + 1}{n}  \right) 2\sqrt{2 \Sigma_{ii}}\, D (1 + \epsilon) \psi(n) \dfrac{(n\log n)^{1/2}}{n}  \ds \sum_{k=1}^{b_n}  k |\Delta_2 w_n(k)|\\
& < \left(\dfrac{n - b_n + 1}{n}  \right) 2\sqrt{2 \Sigma_{ii}}\, D (1 + \epsilon) \psi(n) \dfrac{(\log n)^{1/2}}{n^{1/2}} b_n \ds \sum_{k=1}^{b_n}   |\Delta_2 w_n(k)|\\
& < \left(\dfrac{n - b_n + 1}{n}  \right) 2\sqrt{2 \Sigma_{ii}}\, D (1 + \epsilon) \psi(n) (b_n \log n)^{1/2} \dfrac{b_n^{1/2}}{n^{1/2}} \ds \sum_{k=1}^{b_n}   |\Delta_2 w_n(k)|\\
& \to 0 \text{ by Condition \ref{cond:bn} and \eqref{assbn1}} \; .
\end{align*}

\item Next,
\begin{align*}
& \dfrac{1}{n} \ds \sum_{l=0}^{n-b_n} \ds \sum_{k=1}^{b_n} k^2 |\Delta_2 w_n(k)| \left|\left(\bar{Y}^{(i)}_n - \bar{C}^{(i)}_n \right) \left(\bar{Y}^{(j)}_n - \bar{C}^{(j)}_n \right)   \right| \\ 
& < \dfrac{1}{n} \ds \sum_{l=0}^{n-b_n} \ds \sum_{k=1}^{b_n} k^2 |\Delta_2 w_n(k)| \dfrac{1}{n^2} D^2 \psi(n)^2 \quad \quad \text{by \eqref{yn}}\\
& \leq \left(\dfrac{n-b_n+1}{n}  \right) D^2 \dfrac{b_n^2}{n^2} \psi(n)^2 \ds  \sum_{k=1}^{b_n}    |\Delta_2 w_n(k)| \\
& \to 0 \text{ by Condition \ref{cond:bn} and \eqref{assbn2}} \; .
\end{align*}

\item 
\begin{align*}
& \dfrac{1}{n} \ds \sum_{l=0}^{n-b_n} \ds \sum_{k=1}^{b_n} k^2 |\Delta_2 w_n(k)| \left|\left(\bar{Y}^{(i)}_n - \bar{C}^{(i)}_n \right) \left(\bar{Y}^{(j)}_l(k) - \bar{C}^{(j)}_l(k) \right)  \right|\\
& < \dfrac{1}{n} \ds \sum_{l=0}^{n-b_n} \ds \sum_{k=1}^{b_n} k |\Delta_2 w_n(k)| \left(\dfrac{1}{n} D \psi(n) \right) \left(2D \psi(n)  \right)\\
& < \left(\dfrac{n-b_n + 1}{n}   \right) 2 D^2 \psi(n)^2 \dfrac{1}{n} \ds \sum_{k=1}^{b_n}k |\Delta_2 w_n(k)|\\
& < \left(\dfrac{n-b_n + 1}{n}   \right) 2 D^2 \psi(n)^2 \dfrac{b_n}{n} \ds \sum_{k=1}^{b_n} |\Delta_2 w_n(k)|\\
& \to 0 \text{ by Condition \ref{cond:bn} and \eqref{assbn2}}\, .
\end{align*}

\item 
\begin{align*}
& \frac{1}{n} \ds \sum_{l=0}^{n-b_n} \ds \sum_{k=1}^{b_n} k^2 |\Delta_2 w_n(k)| \left|\left(\bar{Y}^{(i)}_n - \bar{C}^{(i)}_n \right) \left(\bar{C}^{(j)}_l(k) - \bar{C}^{(j)}_n \right) \right| \\
&  \leq \dfrac{1}{n} \ds \sum_{l=0}^{n-b_n} \ds \sum_{k=1}^{b_n} k^2 |\Delta_2 w_n(k)| \left|\left(\bar{Y}^{(i)}_n - \bar{C}^{(i)}_n \right)\bar{C}^{(j)}_l(k) \right| \\
& \quad + \dfrac{1}{n} \ds \sum_{l=0}^{n-b_n} \ds \sum_{k=1}^{b_n} k^2 |\Delta_2 w_n(k)| \left|\left(\bar{Y}^{(i)}_n - \bar{C}^{(i)}_n \right)\bar{C}^{(j)}_n \right|.
\end{align*}

We will show that each of the two terms goes to 0 with probability 1 as $n\to \infty$. 
\begin{align*}
& \dfrac{1}{n} \ds \sum_{l=0}^{n-b_n} \ds \sum_{k=1}^{b_n} k^2 |\Delta_2 w_n(k)| \left|\left(\bar{Y}^{(i)}_n - \bar{C}^{(i)}_n \right)\bar{C}^{(j)}_l(k) \right| \\
& < \dfrac{1}{n} \ds \sum_{l=0}^{n-b_n} \ds \sum_{k=1}^{b_n} k^2 |\Delta_2 w_n(k)|\left(\dfrac{1}{n} D \psi(n) \right) \left(2 (1+ \epsilon) \sqrt{b_n \Sigma_{ii}} \dfrac{1}{k} (\log n)^{1/2}  \right)\quad  \text{by \eqref{diff.bound} and \eqref{yn}}  \\
& < \left(\dfrac{n - b_n + 1}{n}\right)  2 D ( 1 + \epsilon) \sqrt{\Sigma_{ii}} \psi(n) \dfrac{\sqrt{b_n \log n}}{n}\, \sum_{k=1}^{b_n} k |\Delta_2 w_n(k)|\\
& \leq \left(\dfrac{n - b_n + 1}{n}\right)  2 D ( 1 + \epsilon) \sqrt{\Sigma_{ii}} \dfrac{b_n}{n}  \sqrt{b_n \log n}\, \psi(n) \sum_{k=1}^{b_n} |\Delta_2 w_n(k)|\\
& \to 0 \text{ by Condition \ref{cond:bn} and \eqref{assbn1}}\, .
\end{align*}

For the second term,
\begin{align*}
& \dfrac{1}{n} \ds \sum_{l=0}^{n-b_n} \ds \sum_{k=1}^{b_n} k^2 |\Delta_2 w_n(k)| \left|\left(\bar{Y}^{(i)}_n - \bar{C}^{(i)}_n \right)\bar{C}^{(j)}_n \right|\\
& = \dfrac{1}{n} \ds \sum_{l=0}^{n-b_n} \ds \sum_{k=1}^{b_n} k^2 |\Delta_2 w_n(k)| \left|\left(\bar{Y}^{(i)}_n - \bar{C}^{(i)}_n \right) \right| \left|\bar{C}^{(j)}_n \right|\\
& < \dfrac{1}{n} \ds \sum_{l=0}^{n-b_n} \ds \sum_{k=1}^{b_n} k^2 |\Delta_2 w_n(k)| \left(\dfrac{1}{n} D \psi(n)  \right) \left(\dfrac{1}{n} (1 + \epsilon)\left[2n\Sigma_{ii} \log \log n  \right]^{1/2} \right) \quad  \text{ by \eqref{yn} and \eqref{n.bound}}\\
& \leq \left( \dfrac{n-b_n+1}{n} \right)  \sqrt{2 \Sigma_{ii}} D (1+\epsilon) \dfrac{\sqrt{n \log n} \, \psi(n)}{n^2}  \ds \sum_{k=1}^{b_n} k^2 |\Delta_2 w_n(k)|\\
& \leq \left( \dfrac{n-b_n+1}{n} \right)  \sqrt{2 \Sigma_{ii}} D (1+\epsilon) \dfrac{b_n^2 \sqrt{n \log n} \, \psi(n)}{n^2}  \ds \sum_{k=1}^{b_n} |\Delta_2 w_n(k)|\\
&  \leq \left( \dfrac{n-b_n+1}{n} \right)  \sqrt{2 \Sigma_{ii}} D (1+\epsilon) \dfrac{b_n}{n} \dfrac{b_n^{1/2}}{n^{1/2}} (b_n \log n)^{1/2} \psi(n) \ds \sum_{k=1}^{b_n}  |\Delta_2 w_n(k)|\\
& \to 0 \text{ by Condition \ref{cond:bn} and \eqref{assbn1}}\, .
\end{align*}

\item 
\[
\dfrac{1}{n} \ds \sum_{l=0}^{n-b_n} \ds \sum_{k=1}^{b_n} k^2 |\Delta_2 w_n(k)|\left|\left(\bar{C}^{(i)}_l(k) - \bar{C}^{(i)}_n \right)\left(\bar{Y}^{(i)}_l(k) - \bar{C}^{(j)}_l(k)  \right)   \right|.
\]

This term is the same as term 4 except for a change of
components. Thus the same argument can be used to show that it
converges to 0 with probability 1 as $n \to \infty$.

\item 
\[
\dfrac{1}{n} \ds \sum_{l=0}^{n-b_n} \ds \sum_{k=1}^{b_n} k^2 |\Delta_2 w_n(k)| \left|\left(\bar{C}^{(i)}_l(k) - \bar{C}^{(i)}_n \right)\left(\bar{Y}^{(j)}_n - \bar{C}^{(j)}_n  \right)  \right|.
\]

This term is the same as term 7 except for a change of
components. Thus the same argument can be used to show that it
converges to 0 w.p. 1 as $n \to \infty$.
\end{enumerate}
Since each of the nine terms converges to 0 with probability 1, $| \widehat{\Sigma}_{ij} - \Sigma_{ij}| \to 0 $ as $n \to \infty$ with probability 1. 
\end{proof}
Since we proved that $\widehat{\Sigma}_S = \widehat{\Sigma}_{w,n} + d_n \to \Sigma + 0$ as $n \to \infty$ with probability 1, we have the desired result for Theorem \ref{thm:main}.

\subsection{Proof of Theorem~\ref{cor:polymc}} 
\label{sec:appendixB}

Let $S = \{S_{t}\}_{t\geq 1}$ be a strictly stationary stochastic
process on a probability space $(\Omega, {\mathcal F}, P)$ and set
${\mathcal F}_{k}^{l} = \sigma(S_{k}, \ldots, S_{l})$.  Define the
$\alpha$-mixing coefficients for $n=1,2,3,\ldots$ as
\[ 
\alpha(n) = \sup_{k \ge 1} \sup_{A \in {\mathcal F}_{1}^{k}, \, B \in
  {\mathcal F}_{k+n}^{\infty}} | P(A \cap B) - P(A) P(B) | \; . 
\]
The process $S$ is said to be \textit{strongly mixing} if $\alpha(n)
\to 0$ as $n \to \infty$.  It is easy to see that Harris ergodic
Markov chains are strongly mixing; see, for example, \citet{jone:2004}.

\begin{thm}
  \label{thm:kuelbs}
  \citep{kuel:phil:1980} Let $f(S_1), f(S_2), \ldots$ be an
  $\mathbb{R}^p$-valued stationary process such that $\E_F \|f\|^{2 +
    \delta} < \infty$ for some $0 < \delta \le 1$. Let $\alpha_{f}(n)$ be
  the mixing coefficients of the process $\{f(S_t)\}_{t \geq 1}$ and
  suppose, as $n \to \infty$,
\[ 
\alpha_f(n) = O\left(n^{-(1+\epsilon)(1 + 2/\delta)}\right) \quad \text{ for } \,  \epsilon > 0. 
\]
Then there exists a $p$-vector $\theta_f$, a $p \times p$ lower
triangular matrix $L_f$, and a finite random variable $D_f$, such that,
with probability 1,
\begin{equation}
\label{eq:multi_sip_f}
\left\| \sum_{t=1}^{n} f(X_{t}) - n\theta_f - L_f B(n) 
\right\| < D_f\, n^{1/2-\lambda_f}
\end{equation}
for some $\lambda_f > 0$ depending on $\epsilon$, $\delta$, and $p$ only.
\end{thm}

\begin{corollary}
\label{cor:poly_sip}
Let $\E_F \|f\|^{2+\delta} < \infty$ for some $\delta > 0$. If
$X$ is a polynomially ergodic Markov chain of order $\xi \geq (1+
\epsilon)(1 + 2/\delta)$ for some $\epsilon > 0$, then
\eqref{eq:multi_sip_f} holds for any initial distribution.
\end{corollary}

\begin{proof}
  Let $\alpha$ be the mixing coefficient for the Markov chain
  $X=\{X_{t}\}_{t\geq 1}$ and $\alpha_f$ be the mixing coefficient for
  the mapped process $\{f(X_{t})\}_{t\geq1}$.  Then the elementary
  properties of sigma-algebras \citep[cf.][p. 16]{chow:teic:1978}
  shows that $\alpha_{f}(n) \le \alpha(n)$ for all $n$.  Since $X$ is
  polynomially ergodic of order $\xi$ we also have that $\alpha(n) \leq
  \E_FM n^{-\xi}$ for all $n$ and hence if $\xi \geq (1+ \epsilon)(1 +
  2/\delta)$, then $\alpha_f(n) \leq \E_FM n^{-\xi} = O(n^{-(1+
    \epsilon)(1 + 2/\delta)})$. The result follows from Theorem
  \ref{thm:kuelbs} and thus the strong invariance principle as stated,
  holds at stationarity. A standard Markov chain argument (see,
  e.g. Proposition 17.1.6 in \cite{meyn:twee:2009}) shows that if the
  result holds for any initial distribution, then it holds for every
  initial distribution.
\end{proof}

 \begin{proof}[Proof of Theorem~\ref{cor:polymc}]
   Since $\E_F \|g\|^{4 + \delta} <\infty$ implies $\E_F
   \|g\|^{2+\delta} <\infty$ and $X$ is a polynomially ergodic Markov
   chain of order $\xi \geq (1+ \epsilon)(1 + 2/\delta)$ we have from
   Corollary \ref{cor:poly_sip} that an SIP holds such that
\begin{equation*}
\left\| \sum_{t=1}^{n} g(X_{t}) - n\theta - LB(n) 
\right\| < D \,n^{1/2 - \lambda_g} .
\end{equation*}
for some $\lambda_g > 0$ depending on $\epsilon$, $\delta$, and $p$
only.

   Since $\E_F \|g\|^{4 + \delta} <\infty$ implies $\E_F \|h\|^{2+
     \delta} <\infty$ and $X$ is a polynomially ergodic Markov chain
   of order $\xi \geq (1+ \epsilon)(1 + 2/\delta)$ we have from
   Corollary \ref{cor:poly_sip} that an SIP holds such that 
\begin{equation*}
\left\| \sum_{t=1}^{n} h(X_{t}) - n\theta_{h} - L_{h}B(n) 
\right\| < D_{h} \,n^{1/2 - \lambda_h} .
\end{equation*}
for some $\lambda_h > 0$ depending on $\epsilon$, $\delta$, and $p$
only.

   Setting $\lambda=\min\{\lambda_g, \lambda_h\}$ shows that \eqref{eq:multi_sip} and \eqref{eq:multi_sip_h} hold with
\[
\psi(n) = \psi_{h}(n) = n^{1/2 - \lambda}\; .
\]
The rest now follows easily from Theorem~\ref{thm:main}.
\end{proof}
\end{appendix}

\bibliographystyle{apalike}
\bibliography{mcref}

\end{document}